\documentclass[a4paper,10pt,reqno]{amsart}

\usepackage[english]{babel}
\usepackage[utf8]{inputenc}

\usepackage{amssymb}
\usepackage{amsmath}
\usepackage{amsthm}
\usepackage{bbm}

\usepackage{enumitem}
\usepackage{tikz}
\usepackage{marginnote}
\usepackage{color}

\usepackage{csquotes}  % For \enquotes
\newcommand\q{\enquote}

\usepackage{hyperref}   %Backref gives the numbers of pages, where the bibentries are cited
\hypersetup{
    colorlinks=true,
    linkcolor=blue,
    filecolor=magenta,      
    urlcolor=blue,
		citecolor=red,
    pdftitle={Sharelatex Example},
}

\newcommand{\A}{\mathbb{A}}
\newcommand{\C}{\mathbb{C}}
\newcommand{\N}{\mathbb{N}}

\newcommand{\R}{\mathbb{R}}

\newcommand{\Z}{\mathbb{Z}}

\newcommand{\calB}{\mathcal{B}}

\newcommand{\calK}{\mathcal{K}}
\newcommand{\calL}{\mathcal{L}}

\DeclareMathOperator{\id}{id} % identity operator on a vector space
\DeclareMathOperator{\one}{\mathbbm{1}} % constant function with value 1
\DeclareMathOperator{\re}{Re} % real part
\DeclareMathOperator{\dist}{dist} % distance (for instance between a point and a set)
\newcommand{\argument}{\mathord{\,\cdot\,}} % a dot as an argument of a function
\newcommand{\dx}{\;\mathrm{d}} % differential
\newcommand{\norm}[1]{\left\lVert #1 \right\rVert} % norm
\newcommand{\modulus}[1]{\left\lvert #1 \right\rvert} % modulus of real numbers, complex numbers or elements of vector lattices
\newcommand{\dom}[1]{\operatorname{dom} \left( #1 \right)} % domain of linear operators
\newcommand{\inner}[2]{(#1|#2)}

\newcommand{\intt}{{\rm int}\,}
\newcommand{\qintt}{{\rm qint}\,}

\newcommand{\sa}{{\operatorname{sa}}} % for subspace of self-adjoint operators

\newcommand{\spec}{\sigma} % spectrum
\newcommand{\specEss}{\sigma_{\operatorname{ess}}} % essential spectrum
\newcommand{\spr}{r} % spectral radius
\newcommand{\sprEss}{r_{\operatorname{ess}}} % essential spectral radius
\newcommand{\resSet}{\rho} % resolvent set
\newcommand{\Res}{\mathcal{R}} % resolvent
\newcommand{\spb}{s} % spectal bound
\newcommand{\realspb}{\spb_{\R}} % rea spectral bound
\newcommand{\spbEss}{s_{\operatorname{ess}}} % essential spectal bound

\newcommand{\Implies}[2]{``\ref{#1} $\Rightarrow$ \ref{#2}''}

\theoremstyle{definition}
\newtheorem{definition}{Definition}[section]
\newtheorem{remark}[definition]{Remark}

\newtheorem{example}[definition]{Example}
\newtheorem{examples}[definition]{Examples}

\theoremstyle{plain}
\newtheorem{proposition}[definition]{Proposition}
\newtheorem{lemma}[definition]{Lemma}
\newtheorem{theorem}[definition]{Theorem}
\newtheorem{corollary}[definition]{Corollary}

\numberwithin{equation}{section}

\begin{document}

%\title[Stability of positive linear systems]{Stability criteria for positive linear systems -- the continuous-time case}

%
%\sideremark{JG: Ich glaube, du wolltest den Titel noch ändern, Andrii, richtig?
%Was hältst du von dem neuen Vorschlag?}
%
\title[Stability of positive semigroups]{Stability criteria for positive semigroups on ordered Banach spaces}

\author{Jochen Gl\"uck}
\address[J. Gl\"uck]{Bergische Universit\"at Wuppertal, Fakult\"at f\"ur Mathematik und Naturwissenschaften, Gaußstr.\ 20, 42119 Wuppertal, Germany}
\email{glueck@uni-wuppertal.de}
%\thanks{}

\author{Andrii Mironchenko}
\address[A. Mironchenko]{University of
  Klagenfurt, Department of Mathematics, Universit\"atsstraße 65-67, 9020 Klagenfurt, Austria}
\email{andrii.mironchenko@aau.at}
%\thanks{Andrii Mironchenko was supported by the German Research Foundation (DFG) within the project MI 1886/2-2.}

%
\subjclass[2010]{47B65, 47D06, 47A10, 37L15}
%	47B65  	Positive linear operators and order-bounded operators
%	47D06  	One-parameter semigroups and linear evolution equations
%	47A10  	Spectrum, resolvent
% 	37L15  	Stability problems for infinite-dimensional dissipative dynamical systems

%
\keywords{positive systems, continuous-time systems, stability, small-gain condition, linear systems, semigroup theory, resolvent positive operator, Krein--Rutman theorem}
\date{\today}
\begin{abstract}
	We consider generators of positive $C_0$-semigroups and, more generally, 
	resolvent positive operators $A$ on ordered Banach spaces
	and seek for conditions ensuring the negativity of their spectral bound $\spb(A)$.
	Our main result characterizes $\spb(A) < 0$ in terms 
	of so-called \emph{small-gain conditions} 
	that describe the behaviour of $Ax$ for positive vectors $x$.
	This is new even in case that the underlying space is an $L^p$-space
	or a space of continuous functions.
	
	We also demonstrate that it becomes considerably easier to characterize
	the property $\spb(A) < 0$ 
	if the cone of the underlying Banach space has non-empty interior 
	or if the essential spectral bound of $A$ is negative.
	To treat the latter case, we discuss a counterpart of a Krein-Rutman theorem for resolvent positive operators.
	When $A$ is the generator of a positive $C_0$-semigroup,
	our results can be interpreted as stability results for the semigroup,
	and as such, they complement similar results recently proved for the discrete-time case.
	
	In the same vein, we prove a Collatz--Wielandt type formula and a logarithmic formula 
	for the spectral bound of generators of positive semigroups.
\end{abstract}

\maketitle

\section{Introduction} \label{section:introduction}

Recently in \cite{GlM21} it was shown that linear positive discrete-time evolution equations are exponentially stable if and only if the operator generating the evolution equation has a kind of a uniform no-increase property.  
This result has been already applied for the analysis of infinite networks \cite{MKG20, KMS21} and the construction of Lyapunov functions for composite systems \cite{KMS21}. 
Furthermore, it provided strong motivation for the analysis of discrete-time nonlinear monotone systems \cite{MKG20}. 
For finite-dimensional linear positive systems such criteria are well-known, see, e.g., \cite[Section~1]{Rue10} for the 
discrete-time case and \cite{Stern1982} for the continuous-time case.

Although there is a vast literature on the stability of strongly continuous semigroups -- we refer for instance to the classical monographs \cite{vanNeerven1996}, \cite[Chapter~V]{EngelNagel2000}, \cite{Emelyanov2007} and \cite{Eisner2010}, as well as to the survey article \cite{ChillTomilov2007} -- characterizations of stability of positive $C_0$-semigroups by means of no-increase properties have, to the best of our knowledge, only appeared in \cite{Koshkin2015} so far, under quite different assumptions than the ones that we consider in this paper.
We derive several results of this kind, which strongly extend the corresponding finite-dimensional results from \cite{Stern1982}.

A considerable part of the theory of positive $C_0$-semigroups is developed in the setting of Banach lattices. 
Yet, we formulate our results in the more general framework of ordered Banach spaces. 
On the one hand, this gives a wider range of applicability, for instance to positive semigroups that act on $C^*$-algebras 
or on Sobolev spaces. 
On the other hand, recent contributions in the theory of positive dynamical systems show that, 
even if one starts with a positive $C_0$-semigroup on a Banach lattice $X$, 
one typically has to leave the class of Banach lattices if one extends the order to the extrapolation space $X_{-1}$ 
-- an object which occurs, for instance, in the perturbation theory of positive semigroups 
\cite{BarbieriEngel2024Preprint, BatkaiJacobVoigtWintermayr2018} and in positive systems theory \cite{AroraGlueckPaunonenSchwenningerPreprint, Gantouh2023b, Gantouh2024}.

\subsection*{Stability of $C_0$-semigroups}

The most prototypical linear autonomous and con\-ti\-nu\-ous-time dynamical system is described by a $C_0$-semigroup -- which we denote by $(e^{tA})_{t \ge 0}$ or by $(T(t))_{t \ge 0}$ -- with the generator $A$ on a Banach space $X$. 
For an overview of the theory of $C_0$-semigroups, we refer for instance to the monographs \cite{EngelNagel2000, Pazy1983}.

We are interested in the question whether the $C_0$-semigroup converges to $0$ with respect to the operator norm as $t \to \infty$, i.e., whether $\norm{e^{tA}} \to 0$ as $t \to \infty$. In this case $\|e^{tA}\|\leq Me^{-at}$ for some $a,M>0$ and all $t\geq 0$, so the semigroup is said to be \emph{uniformly exponentially stable}. 
A necessary condition for the uniform exponential stability of a $C_0$-semigroup 
is negativity of the \emph{spectral bound} of the generator $A$, i.e., the condition
\begin{align}
	\label{eq:negativity-of-spb}
	\spb(A) := \sup\left\{ \re \lambda: \; \lambda \in \spec(A) \right\} < 0.
\end{align}
This condition is in general not sufficient for uniform exponentially stability.
On the other hand, for many important classes of semigroups, such as eventually norm-continuous semigroups, 
\eqref{eq:negativity-of-spb} is indeed equivalent to uniform exponential stability of the semigroup; 
see Subsection~\ref{sec:stability-of-semigroups} for more details.

\subsection*{Contributions}

In this paper, we present a variety of characterizations for the property $s(A)<0$ and related results under positivity assumptions.
For the sake of generality, we do not restrict ourselves to the infinitesimal generators of positive strongly continuous semigroups but prove the results for the more general class of \emph{resolvent positive operators}, originally introduced in \cite{Arendt1987}.

After some preparations in Section~\ref{section:basics}, 
we present in Section~\ref{section:stability} a general characterization of the negativity of the spectral bound 
of resolvent positive operators in ordered Banach spaces with a normal and generating cone.
Next, in Section~\ref{section:interior-points}, we derive several further characterizations 
in case if the cone has in addition non-empty interior. 
In Section~\ref{section:quasi-compact}, we discuss a Krein--Rutman type theorem for resolvent positive operators 
and use it to derive characterizations for the negativity of spectral bound for operators possessing negative essential spectral bound.
A Collatz--Wielandt formula for the generators of a class of positive $C_0$-semigroups is proved 
in Section~\ref{section:collatz-wielandt}, 
and Section~\ref{section:logarithmic} contains a number of explicit logarithmic formulas for the spectral bound
of generators of positive $C_0$-semigroups.
In Appendix~\ref{section:hilbert-non-positive}, we prove a new characterization of uniform exponential stability 
of general $C_0$-semigroups on Hilbert spaces (without any positivity assumption) 
and discuss how this is related to the main part of the paper.

Discrete-time counterparts of the results proved in 
Sections~\ref{section:stability}, \ref{section:interior-points}, and \ref{section:quasi-compact}
have been established in \cite{GlM21}. 
The differences in the formulations of the results are briefly explained in Section~\ref{section:time-continuous-discussion}.
Some of our arguments are related to the techniques used by Karlin in his classical study of positive operators \cite{Karlin1959}.

\subsection*{Notation and terminology}

We use the conventions $\N = \{1,2,3,\dots\}$ and $\Z_+ = \{0,1,2,\dots\}$.
The identity operator on a Banach space will be denoted by $\id$ (if the space is clear from the context). 
For subsets $A,B$ of a vector space $X$ we denote $A+B:=\{a+b:a\in A,\ b \in B\}$, $-B:=\{-b:b\in B\}$, and $A-B:=A+(-B)$.

\section{Ordered Banach spaces, positive semigroups and resolvent positive operators}
\label{section:basics}

\subsection{Ordered Banach spaces}
\label{subsection:ordered-banach-spaces}

By an \emph{ordered Banach space} we mean a pair $(X, X^+)$, where $X$ is a real Banach space and $X^+$ is a closed cone in $X$, i.e., a closed convex subset of $X$ such that $X^+ + X^+ \subseteq X^+$ and $X^+ \cap (-X^+) = \{0\}$. 
We call $X^+$ the \emph{positive cone} in $X$.
On every ordered Banach space there is a natural order relation $\le$ which is given by $x \le y$ if and only if $y-x \in X^+$. 
The elements of a positive cone are called \emph{positive} vectors.
The positive cone $X^+$ in an ordered Banach space $(X,X^+)$ is called \emph{total} if its linear span $X^+ - X^+$ is dense in $X$; the cone is called \emph{generating} if its linear span is even equal to the whole space $X$. In this case, it can be shown (see for instance \cite[Theorem~2.37]{AliprantisTourky2007}) that there even exists a real number $M > 0$ such that each $x \in X$ can be written as
\begin{align}
	\label{eq:bounded-decomposability-constant}
	x = y-z \quad \text{for vectors} \quad y,z \in X^+ \quad \text{that satisfy} \quad \norm{y}, \norm{z} \le M \norm{x}.
\end{align}
The positive cone $X^+$ in $X$ is called \emph{normal} if there exists a real number $C > 0$ such that
\begin{align*}
	\norm{x} \le C \norm{y}
\end{align*}
whenever $0 \le x \le y$ for vectors $x,y \in X$. For $x,z \in X$ the set $[x,z] := \{y \in X: \, x \le y \le z\}$ is called the \emph{order interval} between $x$ and $z$. It can be shown (see for instance \cite[Theorems~2.38 and~2.40]{AliprantisTourky2007}) that the following three assertions are equivalent:
\begin{enumerate}[label=(\roman*)]
	\item the positive cone is normal
	\item there exists a real number $C > 0$ such that
	\begin{align}
		\label{eq:normality-estimate}
		\norm{x} \le C \max\left\{ \norm{a}, \norm{b} \right\} \quad \text{for all } x,a,b \in X \text{ satisfying } x \in [a,b];
	\end{align}
	\item each order interval in $X$ is norm bounded.
\end{enumerate}

Concise lists of classical examples of ordered Banach spaces can, for instance, be found in \cite[Subsection~2.3]{ArendtNittka2009}, \cite[Subsection~2.3]{GlueckWeber2020} and \cite[Section~2]{GlM21}.

\subsection{Positive operators, Banach spaces and their duals}

Let $(X,X^+)$ be an ordered Banach space, and let $\calL(X)$ denote the space of all bounded linear operators on $X$. An operator $T \in \calL(X)$ is called \emph{positive} -- which we denote by $T \ge 0$ -- if $TX^+ \subseteq X^+$. 
Similarly, a bounded linear functional $x'$ -- i.e.\ an element of the dual space $X'$ -- is called \emph{positive} if $\langle x', x \rangle \ge 0$ for all $x \in X^+$; here, we used the common notation $\langle x', x \rangle := x'(x)$. 
If the cone $X^+$ is total, then the dual space $X'$ also becomes an ordered Banach space when endowed with the \emph{dual cone} $(X')^+$ that is defined to be the set of all positive bounded linear functionals on $X$. 
The dual cone $(X')^+$ is generating if and only if $X^+$ is normal \cite[Theorem~4.5]{KrasnoselskiiLifshitsSobolev1989}; 
and analogously, the dual cone $(X')^+$ is normal if and only if $X^+$ is generating \cite[Theorem~4.6]{KrasnoselskiiLifshitsSobolev1989}.

\subsection{Resolvent positive operators}

Each real Banach space $X$ has a \emph{complexification} which is a complex Banach space that is often denoted by $X_\C$ (in fact, there are many complexifications of $X$, but they are all isomorphic). For an overview about complexifications we refer, for instance, to \cite{MunozSarantopoulosTonge1999} and \cite[Appendix~C]{Glueck2016}. Complexifications are typically used to exploit spectral theoretic properties of linear operators that are a priori defined on real Banach spaces; a brief overview of this approach is, for instance, given in \cite[Section~C.3]{Glueck2016}. 

If $(X,X^+)$ is an ordered Banach space, we call a bounded linear operator $T$ on $X_\C$ \emph{positive} if it is the extension of a positive operator in $\calL(X)$ to $X_\C$; this is equivalent to saying that $T$ leaves $X$ invariant and its restriction to $X$ is positive.

Now, let $(X,X^+)$ be an ordered Banach space and let $A: X \supseteq \dom{A} \to X$ be a closed linear operator. 
Whenever we talk about spectral properties of $A$, we shall assume that $A$ has been extended to a (automatically closed) linear operator $A_\C$ on a complexification $X_\C$ of $X$ and by any spectral property of $A$ we tacitly refer to the corresponding spectral property of $A_\C$. 
In particular, by the \emph{spectrum} $\sigma(A)$ and the \emph{resolvent set} $\resSet(A)$ of $A$ 
we mean the spectrum $\sigma(A_\C)$ and the resolvent set $\resSet(A_\C)$ of $A_\C$. 
For every $\lambda \in \resSet(A)$ the \emph{resolvent} of $A$ at $\lambda$ is the operator 
$\Res(\lambda,A) := \Res(\lambda, A_\C) := (\lambda - A_\C)^{-1} \in \calL(X_\C)$.

An operator $A: X \supseteq \dom{A} \to X$ is called \emph{resolvent positive} if there exists a real number $\omega$ such that the interval $(\omega,\infty)$ is in the resolvent set of $A$ and the resolvent $\Res(\lambda,A)$ is positive for each $\lambda \in (\omega,\infty)$. 
A $C_0$-semigroup $(e^{tA})_{t\geq 0}$ is called positive if $e^{tA}$ is positive for all $t\geq 0$.

One reason why resolvent positive operators are important is that every generator of a positive $C_0$-semigroup is resolvent positive. 
For more information on positive $C_0$-semigroups we refer to the classical paper \cite{BattyRobinson1984} and, in the more specific context of Banach lattices, also to \cite{Nagel1986} and \cite{BatkaiKramarRhandi2017}. 

On the other hand, there also exist resolvent positive operators which are not generators of $C_0$-semigroups \cite[Section~3]{Arendt1987}.
Yet, even then resolvent positivity of an operator $A$ has consequences for the well-posedness of the Cauchy problem $\dot u(t) = Au(t)$: indeed, under mild assumptions on the space $X$, it follows that every resolvent positive operator generates a once integrated semigroup \cite[Theorem~3.11.7]{ArendtBattyHieberNeubrander2011}, which can be interpreted as a weak form of well-posedness of the Cauchy problem (compare \cite[Corollary~3.2.11]{ArendtBattyHieberNeubrander2011}).

Resolvent positive operators were introduced by Arendt in \cite{Arendt1987}, and are mostly studied on ordered Banach spaces with a generating and normal cone. However, their very definition also makes sense on general ordered Banach spaces, and in several sections below we will encounter situations where interesting results can be shown about resolvent positive operators if the cone is only assumed to be total. 

For a linear operator $A: X \supseteq \dom{A} \to X$ define the \emph{spectral bound} $\spb(A)$ by 
\begin{align*}
	%\label{eq:spb}
	\spb(A) := \sup\{\re \lambda: \, \lambda \in \spec(A) \} \in [-\infty,\infty].
\end{align*}
and the \emph{real spectral bound} by 
\begin{align*}
	%\label{eq:real-spb}
	\realspb(A) := \sup(\spec(A) \cap \R) \in [-\infty,\infty].
\end{align*}
Here we use the convention $\sup(\emptyset)=-\infty$. Clearly, we always have $\realspb(A) \leq s(A)$, and if $A$ is resolvent positive, then $\realspb(A) < \infty$. 
If $A$ generates a $C_0$-semigroup $(e^{tA})_{t \ge 0}$, the number 
\begin{align*}
	\omega(A) 
	:= 
	\inf\{\omega \in \R: \, \exists M \ge 1 \; \forall t \ge 0 \; \norm{e^{tA}} \le Me^{t\omega} \}
	\in [-\infty,\infty)
\end{align*}
is called the \emph{growth bound} of $A$ (or of the semigroup $(e^{tA})_{t \ge 0}$).

The following proposition contains several useful results about resolvent positive operators.

\begin{proposition}
	\label{prop:properties-of-resolvent-positive-operators}
	Let $(X,X^+)$ be an ordered Banach space and let $A: X \supseteq \dom{A} \to X$ be a resolvent positive operator on $X$.
	\begin{enumerate}[label=\upshape(\alph*)]
		\item\label{prop:properties-of-resolvent-positive-operators:itm:positive-and-decreasing}
		If $\realspb(A) < \lambda < \mu$, then $\Res(\lambda,A) \ge \Res(\mu,A) \ge 0$. 
		In other words, the resolvent is positive and decreasing on the interval $(\realspb(A), \infty)$ 
		and thus, in particular, on the interval $(\spb(A), \infty)$.
		
		\item\label{prop:properties-of-resolvent-positive-operators:itm:generating-cone}
		Assume that the cone $X^+$ is generating, and define $\dom{A}^+ := \dom{A} \cap X^+$. Then $\dom{A} = \dom{A}^+ - \dom{A}^+$.
	\end{enumerate}
	Assume now in addition that the cone $X^+$ is generating and normal.
	\begin{enumerate}[resume, label=\upshape(\alph*)]
		\item\label{prop:properties-of-resolvent-positive-operators:itm:spectral-bound}
		The spectral bound satisfies $\spb(A) < \infty$. 
		If $\spec(A) \not= \emptyset$, then $\spb(A)$ is a spectral value of $A$.
		
		\item\label{prop:properties-of-resolvent-positive-operators:itm:spectral-bound-equal-real-spectral-bound}
		One has $\spb(A) = \realspb(A)$.
			
		\item\label{prop:properties-of-resolvent-positive-operators:itm:not-positive-to-the-left} 
		If $\lambda \in \R$ is in the resolvent set of $A$ and $\lambda < \realspb(A) = \spb(A)$, then $\Res(\lambda,A)$ is not positive.
	\end{enumerate}
\end{proposition}

\begin{proof}
	\ref{prop:properties-of-resolvent-positive-operators:itm:positive-and-decreasing}
	We first show that $\Res(\mu,A) \ge 0$ for all $\mu \in (\realspb(A),\infty)$. 
	To this end, let us define the set
	
	\begin{align*}
		I := \Big\{ \lambda \in (\realspb(A),\infty): \; \Res(\mu,A) \ge 0 \text{ for all } \mu \in [\lambda, \infty) \Big\}.
	\end{align*}
	Then $I$ is a non-empty subinterval of $(\realspb(A),\infty)$, and $I$ is closed in $(\realspb(A),\infty)$ by the continuity of the resolvent. It suffices to show that $I$ is also open, so let $\lambda \in I$.
	For all $\mu$ in a sufficiently small left neighbourhood of $\lambda$, we can represent $\Res(\mu,A)$ by means of the Taylor expansion
	\begin{align*}
		\Res(\mu,A) = \sum_{n=0}^\infty (\lambda - \mu)^n \Res(\lambda,A)^{n+1},
	\end{align*}
	so $\Res(\mu,A) \ge 0$ for those $\mu$. This proves that $I$ is indeed open, and hence equal to $(\realspb(A),\infty)$. In particular, the resolvent is positive on the latter interval.
	
	The fact that $\mu \mapsto \Res(\mu,A)$ is decreasing on $(\realspb(A),\infty)$ is now easy to see: the resolvent is analytic, and its derivative at any point $\mu \in (\realspb(A),\infty)$ is given by $-\Res(\mu,A)^2 \le 0$.
	
	\ref{prop:properties-of-resolvent-positive-operators:itm:generating-cone}
	Let $x \in \dom{A}$. Choose a real number $\lambda$ in the resolvent set of $A$ such that $\Res(\lambda,A) \ge 0$.  Since $X^+$ is generating in $X$, we can decompose the vector $(\lambda-A)x$ as $(\lambda-A)x = y-z$ for two vectors $y,z \in X^+$. Hence
	\begin{align*}
		x = \Res(\lambda,A)(\lambda-A)x = \Res(\lambda,A)y - \Res(\lambda,A)z \in \dom{A}^+ - \dom{A}^+,
	\end{align*}
	which proves the assertion.
	
	\ref{prop:properties-of-resolvent-positive-operators:itm:spectral-bound} and~\ref{prop:properties-of-resolvent-positive-operators:itm:not-positive-to-the-left}
	These results can, for instance, be found in \cite[Proposition~3.11.2]{ArendtBattyHieberNeubrander2011} (note that the assumption that the cone be generating is not explicitly mentioned there since the authors of \cite{ArendtBattyHieberNeubrander2011} define the notion \emph{ordered Banach space} in a way that the cone is always generating).
	
	\ref{prop:properties-of-resolvent-positive-operators:itm:spectral-bound-equal-real-spectral-bound} 
	This follows immediately from~\ref{prop:properties-of-resolvent-positive-operators:itm:spectral-bound}.
\end{proof}

The proof of Proposition~\ref{prop:properties-of-resolvent-positive-operators}%
\ref{prop:properties-of-resolvent-positive-operators:itm:positive-and-decreasing} 
is an adaptation of an argument from \cite[Proposition~4.2]{DanersGlueckKennedy2016}.
For the sake of easier reference, we explicitly restate parts~\ref{prop:properties-of-resolvent-positive-operators:itm:positive-and-decreasing} 
and~\ref{prop:properties-of-resolvent-positive-operators:itm:generating-cone} of the previous proposition for the case of positive $C_0$-semigroups:

\begin{corollary}
	\label{cor:properties-of-positive-sg-generators}
	Let $(X,X^+)$ be an ordered Banach space and let $A: X \supseteq \dom{A} \to X$ generate a positive $C_0$-semigroup on $X$. 
	Then $A$ is resolvent positive and hence the following assertions hold.
	\begin{enumerate}[label=\upshape(\alph*)]
		\item\label{cor:properties-of-positive-sg-generators:itm:positive-and-decreasing}
		The inequality $\Res(\lambda,A) \ge \Res(\mu,A) \ge 0$ holds for all real numbers $\lambda < \mu$ in the interval $(\realspb(A), \infty)$ 
		(and thus, in particular, in the interval $(\spb(A), \infty)$).
		
		\item\label{cor:properties-of-positive-sg-generators:itm:generating-cone}
		If $X^+$ is generating, then $\dom{A}^+$ is generating in $\dom{A}$, i.e., $\dom{A} = \dom{A}^+ - \dom{A}^+$.
	\end{enumerate}
\end{corollary}

\begin{proof}
	If $\lambda \in \R$ is located on the right of the growth bound $\omega(A)$, 
	then $\Res(\lambda,A)$ can be represented as the Laplace transform of the semigroup and is hence positive. 
	Therefore, $A$ is resolvent positive. 
	Hence, the claims~\ref{cor:properties-of-positive-sg-generators:itm:positive-and-decreasing} 
	and~\ref{cor:properties-of-positive-sg-generators:itm:generating-cone}
	follow from Proposition~\ref{prop:properties-of-resolvent-positive-operators}.
\end{proof}

To the best of our knowledge, positivity of the resolvent on the right of the spectral bound has only been shown 
in the literature so far if the cone $X_+$ is normal and generating: 
in this case one has $\spb(A) = \realspb(A)$ and one does not need the connectedness argument from the proof of 
Proposition~\ref{prop:properties-of-resolvent-positive-operators}\ref{prop:properties-of-resolvent-positive-operators:itm:positive-and-decreasing}
since the Laplace transform representation of the resolvent converges (as an improper Riemann integral) 
on the right of $\spb(A)$ rather than only on the right of the growth bound \cite[Theorem~2.4.2(2)]{BattyRobinson1984}.

We close this subsection with two examples which show that 
Proposition~\ref{prop:properties-of-resolvent-positive-operators}%
\ref{prop:properties-of-resolvent-positive-operators:itm:spectral-bound}--%
\ref{prop:properties-of-resolvent-positive-operators:itm:not-positive-to-the-left}
does not remain true, in general, if one drops the assumption that the cone $X^+$ is normal.
The examples are adaptions of a classical example of Bonsall \cite[Example~(iv) on pp.\,57--58]{Bonsall1958} 
which shows that the spectral radius of a positive operator need not be a spectral value if the cone is not normal.

\begin{examples}
	\label{exa:bonsall}
	Let $\Omega \subseteq \C$ be a bounded domain, which is symmetric with respect to reflection at the real axis 
	and let $\A(\Omega)$ denote the space of all continuous functions $\overline{\Omega} \to \C$ 
	which are holomorphic inside $\Omega$. 
	This is a complex Banach space with respect to the sup norm. 
	Let $X \subseteq \A(\Omega)$ denote the subset of those functions that are real-valued on $\Omega \cap \R$. 
	Then $X$ is a real Banach space with respect to the sup norm and $\A(\Omega)$ is the complexification of $X$. 
	Indeed, for every $f \in \A(\Omega)$ the function $f^*$ given by $f^*(z) := \overline{f(\overline{z})}$ 
	for all $z \in \overline{\Omega}$ is well-defined due to the symmetry of $\Omega$ and belongs to $\A(\Omega)$, 
	and we have $f = g_1 + i g_2$, where the functions $g_1 := \frac{1}{2}(f+f^*)$ and $g_2 := \frac{1}{2i}(f-f^*)$ are in $X$. 
	Moreover, the intersection $X \cap iX$ is $\{0\}$ by the identity theorem for holomorphic functions. 
	
	Now let $S \subseteq \Omega \cap \R$ be a set which has an accumulation point in $\Omega$ 
	and let $X^+$ consist of those functions in $X$ which map $S$ into $[0,\infty)$. 
	This is a closed cone in $X$ by the identity theorem for holomorphic functions, so $(X,X^+)$ is an ordered Banach space.
	
	It follows from $\one_{\overline{\Omega}} \in X$ that the cone $X^+$ is generating 
	(and even has an interior point, see Lemma~\ref{lem:properties-interior-points}). 
	However, $X^+$ is not normal. 
	To see this, choose a point $z_0 \in \Omega \cap \R$ and define functions $f_n \in X$ by 
	$f_n(z) = \sin\big(n(z-z_0)\big) + 1$ for each $z \in \overline{\Omega}$ and each integer $n \ge 1$.
	Then $0 \le f_n \le 2 \one_{\overline{\Omega}}$ for each $n$, 
	but since $\sin$ is unbounded on $\C$ one has $\sup_n \norm{f_n}_\infty = \infty$.
	
	Now let the operator $A_\C \in \calL(\A(\Omega))$ be given by $(A_\C f)(z) = zf(z)$ for all $f \in \A(\Omega)$ 
	and all $z \in \overline{\Omega}$. 
	Then $A_\C$ leaves $X$ invariant and is the complex extension of the operator $A := A_\C|_{X}$. 
	One can easily check that the spectrum of $A_\C$ -- and hence of $A$ -- equals $\overline{\Omega}$. 
	Moreover, $A$ is positive if and only if $S \subseteq [0,\infty)$.
	
	Let us now consider two more specific situations:
	\begin{enumerate}[label=(\alph*)]
		\item 
		Let $\Omega$ be a subset of the right half plane which satisfies 
		\begin{align*}
			\sup (\Omega \cap \R) < \sup \{\re z: \, z \in \Omega\}.
		\end{align*}
		Then $A$ is positive since $S \subseteq \Omega \cap \R$ is contained in $(0,\infty)$ and 
		hence $A$ is resolvent positive as a consequence of the Neumann series representation of the resolvent. 
		One has $\realspb(A) = \sup (\Omega \cap \R)$ 
		and $\spb(A) = \sup \{\re z: \, z \in \Omega\}$. 
		So $\realspb(A) < \spb(A)$ and $\spb(A)$ is not a spectral value of $A$. 
		This shows that Proposition~\ref{prop:properties-of-resolvent-positive-operators}%
		\ref{prop:properties-of-resolvent-positive-operators:itm:spectral-bound} 
		and~\ref{prop:properties-of-resolvent-positive-operators:itm:spectral-bound-equal-real-spectral-bound} 
		does not remain true, in general, if $X^+$ is not normal.
		
		\item 
		Let $\Omega := \{z \in \C: \, 1 < \modulus{z - 2} < 2\}$ and choose $S := (0,1)$. 
		Then $A$ is positive and hence resolvent positive. 
		Moreover, the number $\spb(A) = \realspb(A) = 4$ is a spectral value of $A$. 
		At the number $2 \in \resSet(A)$ the resolvent of $A$ is given by 
		$\Res(2,A)f(z) = \frac{1}{2-z}f(z)$ for all $f \in \A(\Omega)$ and all $z \in \overline{\Omega}$. 
		Since $\frac{1}{2-z} > 0$ for all $z \in S = (0,1)$, it follows that $\Res(2,A)$ is positive 
		despite $2 < \realspb(A) = \spb(A)$. 
		This shows that 
		Proposition~\ref{prop:properties-of-resolvent-positive-operators}%
		\ref{prop:properties-of-resolvent-positive-operators:itm:not-positive-to-the-left} 
		does not hold, in general, if the cone is not normal.
	\end{enumerate}
\end{examples}

\subsection{Stability of $C_0$-semigroups}
\label{sec:stability-of-semigroups}

The main purpose of this article is to characterise for resolvent positive operators $A$,
in a variety of situations, the property $\spb(A) < 0$. 
The motivation for studying this property is as follows:

Let us first consider a $C_0$-semigroup $(e^{tA})_{t\geq 0}$ with the generator $A$ on a Banach space $X$. 
An important stability property of semigroups is the uniform exponential stability: a semigroup $(e^{tA})_{t\geq 0}$
is called \emph{uniformly exponentially stable}, if there are numbers $M,a>0$, such that 
\[
	\norm{e^{tA}} \leq Me^{-at} \quad \text{for all }  t \geq 0. 
\]
It is not hard to see that uniform exponential stability of the semigroup implies that $\spb(A)<0$; see for instance \cite[Proposition 1.2.1]{vanNeerven1996}. At the same time, the converse implication does not hold in general, even for $C_0$-groups on Hilbert spaces, see \cite[Example 1.2.4]{vanNeerven1996} or \cite[Section~IV.3]{EngelNagel2000}. Even if $\spb(A) = -\infty$, the semigroup does not need to be uniformly exponentially stable \cite[Exercise 4.13]{CuZ20}. This behaviour is due to a failure of the spectral mapping theorem for general $C_0$-semigroups; for a detailed explanation of this, we refer for instance to \cite[Section~IV.3]{EngelNagel2000}.

On the other hand, for many classes of semigroups the spectral bound $\spb(A)$ is equal to the growth bound $\omega(A)$ of the semigroup. 
If this is the case, one says that $T$ satisfies the \emph{spectrum determined growth property} \cite[p. 161]{CuZ20}. 
A notable class of $C_0$-semigroups that satisfy the spectrum determined growth assumption are eventually norm-continuous semigroups \cite[Theorem 1.10, p. 302]{EngelNagel2000}, and thus, as a special case, analytic semigroups. 
For a very general condition that implies equality of the spectral bound and the growth bound, we refer to \cite[Corollary~1.4(i)]{MartinezMazon1996}, where semigroups that are \emph{norm continuous at infinity} are studied.

\subsection{Stability of positive $C_0$-semigroups}
\label{section:stability-comments}

Positivity of semigroups enters the game due to the following two reasons:

\begin{enumerate}[label=(\arabic*)]
	\item 
	On many important spaces, \emph{every} positive semigroup has the spectrum determined growth property.
	This is, for instance, true on $L^p$-spaces for $p \in [1,\infty)$ 
	(see \cite{Weis1995} or \cite[Theorem~1]{Weis1998} or, for a strong simplification of the proof, 
	the recent article \cite{Vogt2022}),
	and also on many spaces of continuous functions
	(see for instance \cite[Theorem~B-IV-1.4]{Nagel1986} or, 
	for a more recent and simpler argument, \cite[Theorem~1]{AroraGlueck2022}). 
	See also Subsection~\ref{section:s_A_equals_omega_A} below.
	
	\item 
	Positivity provides us with ways to show the property $\spb(A) < 0$ 
	without explicitly computing or estimating the spectrum of $A$.	
\end{enumerate}

The focus of the present article is on point~(2).
Our results apply to a slightly more general situation than described above.
We do not only focus on the case where $A$ generates a positive $C_0$-semigroup. 
Instead, we consider the more general situation where $A$ is a resolvent positive operator. 
One advantage of this approach is that resolvent positive operators
are closely related to \emph{integrated semigroups}, a more general concept than $C_0$-semigroups;
just as for the $C_0$-semigroup case, stability properties of integrated semigroups
are strongly tied to the spectrum of their generator
(see for instance \cite[Theorem~6.1]{ElMennaoui1994}), 
which is why the inequality $\spb(A) < 0$ is also of interest for operators $A$
that do not generate $C_0$-semigroups.

\section{Spectral stability of resolvent positive operators}
\label{section:stability}

In this section, we characterize the condition $s(A)<0$ for resolvent positive operators acting on ordered Banach spaces
with generating and normal cones. 
In view of the preceding discussion, this is relevant in order to establish uniform exponential stability of linear systems.
Characterizations under different assumptions on the cone will be given in the subsequent sections.

\subsection{A characterization of spectral stability}

For a subset $S$ and a vector $x$ in a Banach space $X$, we denote by
\begin{align*}
	\dist(x,S) := \inf\left\{ \norm{x-y}: \, y \in S \right\}
\end{align*}
the \emph{distance} from $x$ to $S$. 
We proceed to our first main result:

\begin{theorem} 
	\label{thm:stability-for-pos-sg}
	Let $(X,X^+)$ be an ordered Banach space with generating and normal cone and let $A: X \supseteq \dom{A} \to X$ be a resolvent positive operator. Then the following assertions are equivalent:
	\begin{enumerate}[label=\upshape(\roman*)]
		\item\label{thm:stability-for-pos-sg:itm:stability} 
		\emph{Spectral stability:}
		The spectral bound of $A$ satisfies $\spb(A) = \realspb(A) < 0$.

		\item\label{thm:stability-for-pos-sg:itm:pos-resolvent} 
		\emph{Positive resolvent at $0$:}
		The operator $A: \dom{A} \to X$ is bijective and  
		$-A^{-1} = \Res(0,A)$ is positive.
		
		\item\label{thm:stability-for-pos-sg:itm:mbi} 
		\emph{Monotone bounded invertibility property:}
		There exists a number $c > 0$ such that
		\begin{align*}
			-Ax \le y \qquad \Rightarrow \qquad \norm{x} \le c \norm{y}
		\end{align*}
		for all $0 \le x \in \dom{A}$ and all $0 \le y \in X$.
		
		\item\label{thm:stability-for-pos-sg:itm:uniform-small-gain}
		\emph{Uniform small-gain condition:}
		There exists a number $\eta > 0$ such that
		\begin{align}
		\label{eq:uniform-SGC}
			\dist(Ax,X^+) \ge \eta \norm{x} \quad \text{for all } 0 \le x \in\dom{A}.
		\end{align}
		
		\item\label{thm:stability-for-pos-sg:itm:perturbed-small-gain} 
		\emph{Robust small-gain condition:}
		There exists a number $\varepsilon > 0$ such that
		\begin{align*}
			(A+P)x \not\ge 0
		\end{align*}
		whenever $x$ is a positive non-zero vector in $\dom{A}$ and $P: X \to X$ is a positive linear operator of norm $\norm{P} \le \varepsilon$.
		
		\item\label{thm:stability-for-pos-sg:itm:perturbed-rank-1-small-gain}
		\emph{Rank-$1$ robust small-gain condition:}
		There exists a number $\varepsilon > 0$ such that 
		\begin{align*}
			(A+P)x \not\ge 0
		\end{align*}
		whenever $x$ is a positive non-zero vector in $\dom{A}$ and $P: X \to X$ is a positive linear operator of norm $\norm{P} \le \varepsilon$ and of rank $1$.
	\end{enumerate}
\end{theorem}

For some background information about the terminology \emph{small-gain condition}
we refer to Subsection~\ref{subsection:finite-dim}.

Let $A: X \supseteq \dom{A} \to X$ denote a closed linear operator. 
If $\lambda$ is a scalar and $(x_n) \subseteq \dom{A}$ is a sequence of vectors 
such that $\norm{x_n}_X = 1$ for each $n$ and 
\begin{align*}
	(\lambda - A)x_n \to 0 \qquad \text{in } X,
\end{align*}
then $\lambda$ is called an \emph{approximate eigenvalue} of $A$, 
and $(x_n)$ is called an approximate eigenvector associated to $\lambda$. 
Every approximate eigenvalue of $A$ is in the spectrum of $A$.

To prove Theorem~\ref{thm:stability-for-pos-sg}, we need the following lemma. 
\begin{lemma}
	\label{lem:positive-approximate-eigenvector}
	Let $(X,X^+)$ be an ordered Banach space with generating cone and let $A: X \supseteq \dom{A} \to X$ be a resolvent positive operator. Assume that $A$ has at least one spectral value in $\R$.
	
	Then $\realspb(A)$ is an approximate eigenvalue of $A$, and there exists an associated approximate eigenvector that consists of positive vectors. 
\end{lemma}

The proof is a simple adaptation of \cite[Proposition~IV.1.10]{EngelNagel2000}. In \cite[Lemma~3.4]{GlM21} we gave a similar argument to obtain the lemma for the special case of bounded positive operators (note that normality of the cone that is assumed in this reference, is only needed there to ensure that the spectral radius is in the spectrum). Still, we include the details here to be more self-contained.

\begin{proof}[Proof of Lemma~\ref{lem:positive-approximate-eigenvector}]
	First note that $\realspb(A)$ is a spectral value of $A$ since the spectrum is closed. Now, choose a sequence $(s_n)$ of real numbers such that $s_n \downarrow \realspb(A)$; then we have $\norm{\Res(s_n,A)} \to \infty$. Hence, we can find a sequence of normalized vectors $w_n \in X$ such that $\norm{\Res(s_n,A)w_n} \to \infty$. 
	
	Now we use that the cone is generating: by decomposing the $w_n$ into a difference of positive and uniformly bounded vectors (see~\eqref{eq:bounded-decomposability-constant}), we can even find a bounded sequence $(v_n) \subseteq X^+$ such that $\beta_n := \norm{\Res(s_n,A)v_n} \to \infty$. By dropping finitely many elements of our sequence if necessary, we may assume that $\beta_n > 0$ for each $n$. 
	We now obtain the desired approximate eigenvector $(x_n)$ by setting
	\begin{align*}
		x_n := \frac{1}{\beta_n} \Res(s_n,A)v_n
	\end{align*}
	for each $n$. Clearly, each $x_n$ is normalized, and it is also in the cone $X^+$ since $\Res(s_n,A)$ is positive according to Proposition~\ref{prop:properties-of-resolvent-positive-operators}\ref{prop:properties-of-resolvent-positive-operators:itm:positive-and-decreasing}. Finally, $(x_n)$ is indeed an approximate eigenvector for $\realspb(A)$ since we have
	\begin{align*}
		\big(\realspb(A) - A\big)x_n = \big(\realspb(A) - s_n\big)x_n + \frac{1}{\beta_n} v_n \to 0
	\end{align*}
	as $n \to \infty$; 
	here we used that $\beta_n \to \infty$ and that the sequence $(v_n)$ is bounded.
\end{proof}

Now we can show Theorem~\ref{thm:stability-for-pos-sg}:

\begin{proof}[Proof of Theorem~\ref{thm:stability-for-pos-sg}]
	\Implies{thm:stability-for-pos-sg:itm:stability}{thm:stability-for-pos-sg:itm:pos-resolvent}
	By Proposition~\ref{prop:properties-of-resolvent-positive-operators}\ref{prop:properties-of-resolvent-positive-operators:itm:spectral-bound-equal-real-spectral-bound} one has $\realspb(A) = \spb(A)$.
	Thus, Proposition~\ref{prop:properties-of-resolvent-positive-operators}\ref{prop:properties-of-resolvent-positive-operators:itm:positive-and-decreasing} shows the claim.
	
	\Implies{thm:stability-for-pos-sg:itm:pos-resolvent}{thm:stability-for-pos-sg:itm:mbi}
	If $-Ax \le y$ for $0 \le x \in \dom{A}$ and $0 \le y \in X$, then it follows from the positivity of $\Res(0,A)$ that $0 \le x \le \Res(0,A)y$. So the monotone bounded invertibility property follows from the boundedness of $\Res(0,A)$ and the normality of the cone.
	
	\Implies{thm:stability-for-pos-sg:itm:mbi}{thm:stability-for-pos-sg:itm:uniform-small-gain}
	Let $0 \le x \in \dom{A}$ and let $\varepsilon > 0$. There exists a vector $z \in X^+$ such that
	\begin{align*}
		\dist(Ax, X^+) + \varepsilon \ge \norm{Ax - z}.
	\end{align*}
	We can decompose the vector $Ax - z$ as $Ax - z = u - v$ for positive vectors $u,v$ that satisfy the norm estimate
	\begin{align*}
		\norm{u}, \norm{v} \le M \norm{Ax-z} \le M \dist(Ax, X^+) + M \varepsilon;
	\end{align*}
	here, $M > 0$ is the constant from~\eqref{eq:bounded-decomposability-constant}. Now we can estimate the vector $-Ax$ as
	\begin{align*}
		-Ax = v-u-z \le v,
	\end{align*}
	so the monotone bounded invertibility property implies that
	\begin{align*}
		\norm{x} \le c \norm{v} \le cM \dist(Ax,X^+) + cM \varepsilon.
	\end{align*}
	Since $\varepsilon$ was arbitrary, this implies the uniform small-gain condition with constant $\eta = 1/(cM)$.

	\Implies{thm:stability-for-pos-sg:itm:uniform-small-gain}{thm:stability-for-pos-sg:itm:perturbed-small-gain}
	Set $\varepsilon = \eta/2$. Let $P: X \to X$ be a positive linear operator of norm at most $\varepsilon$ and let $x \in X^+$ be a non-zero vector. Then the distance between $Ax$ and $(A+P)x$ is given by
	\begin{align*}
		\norm{Ax - (A+P)x} = \norm{Px} \le \varepsilon \norm{x} < \eta \norm{x},
	\end{align*}
	where we used $x \not= 0$ for the strict inequality at the end. Hence, $(A+P)x$ cannot be positive due to the uniform small-gain condition.
	
	\Implies{thm:stability-for-pos-sg:itm:perturbed-small-gain}{thm:stability-for-pos-sg:itm:perturbed-rank-1-small-gain}
	This implication is obvious.
	
	\Implies{thm:stability-for-pos-sg:itm:perturbed-rank-1-small-gain}{thm:stability-for-pos-sg:itm:stability}
	Let $\varepsilon > 0$ be as in~\ref{thm:stability-for-pos-sg:itm:perturbed-rank-1-small-gain} and assume towards a contradiction that $\spb(A) \ge 0$. 
	Since the cone $X^+$ is generating and normal, $\spb(A)$ is a spectral value of $A$ (see Proposition~\ref{prop:properties-of-resolvent-positive-operators}\ref{prop:properties-of-resolvent-positive-operators:itm:spectral-bound}) and thus coincides with $\realspb(A)$. 
	Hence, Lemma~\ref{lem:positive-approximate-eigenvector} yields that $\spb(A)$ is an approximate eigenvalue of $A$ and that there exists a corresponding approximate eigenvector $(x_n)$ that consists of vectors $0 \le x_n \in \dom{A}$ such that $\norm{x_n} = 1$ and $(\spb(A) - A)x_n \to 0$ as $n \to \infty$. 
	For each $n$, we find vectors $y_n, z_n \in X^+$ such that
	\begin{align*}
		(A - \spb(A))x_n = y_n - z_n,
	\end{align*}
	and according to~\eqref{eq:bounded-decomposability-constant}, we can choose these vectors such that $y_n, z_n \to 0$ as $n \to \infty$. 
	Now let $M' > 0$ be the constant from~\eqref{eq:bounded-decomposability-constant}, but for the dual space $X'$; 
	this constant exists since the cone in $X^+$ is normal, so the dual cone is generating \cite[Theorem~4.5]{KrasnoselskiiLifshitsSobolev1989}.
	As $z_n \to 0$, there exists an index $n_0$ such that $M'\norm{z_{n_0}} \le \varepsilon$.
	
	It is not difficult to see that there exists a functional $0 \le z' \in X'$ of norm at most $M'$ such that $\langle z', x_{n_0} \rangle \ge 1$ 
	(we refer to \cite[Lemma~3.5]{GlM21} for a detailed proof of this). 
	Now we define a positive rank-$1$ operator $P: X \to X$ by the formula $Pv = \langle z', v \rangle z_{n_0}$ for each $v \in X$.
	
	This operator $P$ has norm $\norm{P} = \norm{z'} \norm{z_{n_0}} \le M' \norm{z_{n_0}} \le \varepsilon$. 
	But on the other hand, as we assumed $\spb(A) \ge 0$, 
	\begin{align*}
		Ax_{n_0} + Px_{n_0} & \ge (A - \spb(A))x_{n_0} + P x_{n_0} \\
		& = y_{n_0} - z_{n_0} + \langle z', x_{n_0} \rangle z_{n_0} \ge -z_{n_0} + z_{n_0} = 0. 
	\end{align*}
	This contradicts the rank-$1$ robust small-gain condition.
\end{proof}

\subsection{The finite-dimensional case revisited}
\label{subsection:finite-dim}

Let us briefly compare Theorem~\ref{thm:stability-for-pos-sg} 
with a classical result in finite dimensions.
Let $(X,X^+)$ be a finite-dimensional ordered Banach space with generating cone
(normality is automatic in finite dimensions)
and let $A: X \to X$ be a linear operator 
such that $e^{tA}$ is positive for each $t \ge 0$
(equivalently: $A$ is resolvent positive).
Stern proved in \cite[Corollary~1.6]{Stern1982} that $\spb(A) < 0$ 
if and only if
\begin{align}
	\label{eq:finite-dim-small-gain-condition}
	Ax \not\geq 0 \quad \text{for all} \quad x \in X^+ \setminus \{0\}.
\end{align}
Due to the occurrence of similar conditions in so-called 
\emph{small-gain theorems} for the stability of networks
(see for instance \cite{JTP94,DRW10,MKG20})
we call~\eqref{eq:finite-dim-small-gain-condition} a \emph{small-gain} condition.

In infinite dimensions, the condition~\eqref{eq:finite-dim-small-gain-condition}
is not sufficient to guarantee that $\spb(A) < 0$,
as each of the following simple examples shows: 

\begin{example}
	\label{exa:shift-non-uniform} 
	\begin{enumerate}[label=(\alph*)]
		\item\label{exa:shift-non-uniform:item:differential} 
		Equip the space $L^2(\R)$ with the pointwise almost everywhere order and
		let $(e^{tA})_{t \ge 0}$ denote the left shift semigroup on $X = L^2(\R)$, i.e.\ 
		$e^{tA}f = f(\argument+t)$ for all $f \in L^2(\R)$ and each $t \ge 0$. 
		The domain of $A$ is the Sobolev space $H^1(\R)$ and one has $Af = f'$ for each $f \in H^1(\R)$. 
		If a function $0 \le f \in H^1(\R)$ satisfies $f' = Af \ge 0$, 
		then $f$ is increasing and hence $f=0$ as $f \in L^2(\R)$. 
		
		Therefore, $Af \not\ge 0$ for all non-zero $0 \le f \in \dom{A}$.
		Yet, one has $\spb(A) = 0$: 
		by using the Fourier transform one can see that $\sigma(A) = i\R$. 
				
		\item\label{exa:shift-non-uniform:item:bounded}
		Endow $X := \ell^2(\N)$ with the pointwise order, let $R \in \calL(X)$ denote the right shift 
		and let $A := R - \frac{1}{2}\id$. 
		As $R$ is positive it is resolvent positive and hence, $A$ is resolvent positive, too. 
		If $x \in X^+$ satifies $Ax \ge 0$, then 
		\begin{align*}
			(0, x_1, x_2, \dots) \ge \frac{1}{2} (x_1, x_2, x_3, \dots),
		\end{align*}
		so $x = 0$. 
		This shows that $A$ satisfies the condition~\eqref{eq:finite-dim-small-gain-condition}. 
		Yet, the spectrum of $R$ is the closed unit disk and hence $\spb(A) = \frac{1}{2} > 0$.
	\end{enumerate}
\end{example}

Yet, Theorem~\ref{thm:stability-for-pos-sg} shows that 
the small-gain condition~\eqref{eq:finite-dim-small-gain-condition}
can be replaced with the uniform small-gain 
condition~\eqref{eq:uniform-SGC}, 
which is indeed equivalent to $\spb(A) < 0$.
It is worthwhile to explain the relation between the conditions~\eqref{eq:uniform-SGC} and~\eqref{eq:finite-dim-small-gain-condition}
at an intuitive level:
for positive $x$, the condition $Ax \not\ge 0$
from the small-gain condition~\eqref{eq:finite-dim-small-gain-condition}
can be rewritten as $\dist(Ax, X^+) > 0$,
and the uniform small-gain condition~\eqref{eq:uniform-SGC} 
is simply a uniform version thereof.

In finite dimensions, it easily follows from a compactness argument 
that the conditions~\eqref{eq:uniform-SGC} and~\eqref{eq:finite-dim-small-gain-condition} are equivalent
(alternatively, this also follows from the above quoted \cite[Corollary~1.6]{Stern1982}
and from Theorem~\ref{thm:stability-for-pos-sg} 
since both conditions are equivalent to $\spb(A) < 0$).

\subsection{Single operators vs.\ $C_0$-semigroups} 
\label{section:time-continuous-discussion}

Theorem~\ref{thm:stability-for-pos-sg} gives characterizations for the negativity of the spectral bound 
of resolvent positive operators and thus in particular for generators of positive $C_0$-semigroups.
Those characterizations are the counterparts for the small-gain criteria of stability of discrete-time semigroups 
shown recently in \cite{GlM21}.
In particular, a counterpart of Theorem~\ref{thm:stability-for-pos-sg} says (among other statements) 
the following (see \cite[Theorem~3.3]{GlM21}): 

\begin{theorem} 
	\label{thm:stability-for-pos-ops-discrete-time}
	Let $(X,X^+)$ be an ordered Banach space with generating and normal cone and assume that $T \in \calL(X)$ 
	is a positive operator. 
	Then the following assertions are equivalent:
	\begin{enumerate}[label=\upshape(\roman*)]
		\item\label{thm:stability-for-pos-ops-discrete-time:itm:stability} 
		The spectral radius of $T$ satisfies $\spr(T) < 1$.

		\item\label{thm:stability-for-pos-ops-discrete-time:itm:mbi} 
		There is a number $c > 0$ such that one has
		\begin{align*}
			(\id - T)x \le y \qquad \Rightarrow \qquad \norm{x} \le c \norm{y}.
		\end{align*}
		for all $x,y \in X^+$.
				
		\item\label{thm:stability-for-pos-ops-discrete-time:itm:uniform-small-gain} 
		There exists a number $\eta > 0$ such that for each $x \in X^+$
		\begin{align*}
			\dist\big((T-\id)x,X^+\big) \ge \eta \norm{x}.
		\end{align*}
	\end{enumerate}
\end{theorem}

Note that, since in Theorem~\ref{thm:stability-for-pos-ops-discrete-time} the cone $X^+$ is generating and normal, the positivity of $T$ implies that $\spr(T)$ is a spectral value of $T$ (see for instance \cite[paragraph~2.2 on p.\ 311]{SchaeferWolff1999}), and hence $\spr(T) = \realspb(T) = \spb(T)$. Hence the condition $\spr(T)<1$ is equivalent to $\spb(T-\id)<0$.

After this reformulation, looking at Theorems~\ref{thm:stability-for-pos-sg} and \ref{thm:stability-for-pos-ops-discrete-time}, 
we see that the statements of Theorems~\ref{thm:stability-for-pos-sg} are \q{translated} into the statements of Theorems~\ref{thm:stability-for-pos-ops-discrete-time} by substituting $T-\id$ instead of $A$.
This difference can be understood by noting that mimicking the definition of the infinitesimal generator of a strongly continuous semigroup, the role of the \q{infinitesimal generator} for the discrete-time semigroup $\{T^k:k\in\Z_+\}$ is played by $T-\id$.

\subsection{Equality of the spectral and growth bound} 
\label{section:s_A_equals_omega_A}

Let $(e^{tA})_{t \ge 0}$ be a positive $C_0$-semigroup on an ordered Banach space with normal and generating cone.
Theorem~\ref{thm:stability-for-pos-sg} characterizes the property $\spb(A) < 0$. 
This is equivalent to $e^{tA} \to 0$ in operator norm as $t \to \infty$ if the semigroup has the spectral determined growth property 
$\spb(A) = \omega(A)$.
As explained in Subsection~\ref{section:stability-comments} this property is, for instance, 
satisfied for positive semigroups on $L^p$-spaces. 
The following theorem gives another sufficient condition for this property. 
A set $S$ in an ordered Banach space $(X,X^+)$ is called \emph{order bounded} if there exist points $x,y \in X$ such that $S \subseteq [x,y]$.

\begin{theorem}
	\label{thm:s_A_equals_omega_A}
	Let $(X,X^+)$ be an ordered Banach space with generating and normal cone 
	and let $A: X \supseteq \dom{A} \to X$ generate a positive $C_0$-semigroup on $X$. 
	Let $0 \le t_0 < t_1$ and assume that the set $\{e^{tA}f: \, t \in [t_0,t_1]\}$ is order bounded for each $f \in X$. 
	Then $\spb(A) = \omega(A)$.
\end{theorem}

\begin{proof} 
	Since replacing $A$ with $A-c\id$ for any number $c \in \R$ does not change the order boundedness assumption, 
	it suffices to prove that if $\spb(A) < 0$, then $\omega(A) \le 0$. 
	So assume that $\spb(A) < 0$. 
	
	First, we will show that the orbit $\{e^{tA}x : \, t \ge 0\}$ is norm bounded for every $x \in X^+$.
	From this fact and the assumption that $X^+$ is generating, it then follows that the orbit of each vector in $X$ is norm bounded 
	and by the uniform boundedness principle, we thus get boundedness of the semigroup, so $\omega(A) \le 0$.
	
	So fix $x \in X^+$.
	We will use the following known result:
	since the cone $X^+$ is assumed to be normal and generating, 
	the resolvent $\Res(0,A)$ is given by 
	\begin{align*}
		\Res(0,A)x = \int_0^\infty e^{sA} x \dx s \quad \text{for each } x\in X,
	\end{align*}
	where the integral converges as an improper Riemann integral \cite[Theorem~2.4.2(2)]{BattyRobinson1984}.
	
	Due to the order boundedness assumption, there exists a vector $y \in X_+$ such that 
	$0 \le e^{tA} x \le y$ for all $t \in [t_0, t_1]$. 
	For every $t \ge t_1$ this implies 
	\begin{align*}
		e^{tA} x 
		= 
		\frac{1}{t_1-t_0}
		\int_{t_0}^{t_1} e^{(t-s)A}e^{sA}x \dx s
		& \le 
		\frac{1}{t_1-t_0} 
		\int_{t_0}^{t_1} e^{(t-s)A} y \dx s 
		\\ 
		& \le 
		\frac{1}{t_1-t_0}
		\int_0^\infty e^{sA} y \dx s
		= 
		\frac{1}{t_1-t_0}
		\Res(0,A)y
		.
	\end{align*}
	As the cone is normal, this implies that $\{e^{tA}x : \, t \ge t_1\}$ is norm bounded 
	and hence, $\{e^{tA}x : \, t \ge 0\}$ is norm bounded, too, as claimed.
\end{proof}

In the special case where $X$ is a Banach lattice and $t_0 = 0$, 
Theorem~\ref{thm:s_A_equals_omega_A} was proved in \cite[Theorem~3.1]{GlueckKaplin2024}. 
Our proof is an adaptation of the proof in this reference, 
which is in turn an adaptation of the proof of \cite[Theorem~1]{AroraGlueck2022}.
Theorem~\ref{thm:s_A_equals_omega_A} contains various known results as special cases:

\begin{examples}
	\label{exas:s_A_equals_omega_A}
	Let $(e^{tA})_{t \ge 0}$ be a positive $C_0$-semigroup on an ordered Banach space $(X,X^+)$ 
	and assume that the cone $X^+$ is normal and generating. 
	Each of the following conditions implies that $\spb(A) = \omega(A)$.
	\begin{enumerate}[label=(\alph*)]
		\item\label{exas:s_A_equals_omega_A:item:non-empty-interior} 
		The cone $X^+$ has non-empty interior. 
		In this case, proofs of $\spb(A) = \omega(A)$ can, for instance, be found in 
		\cite[Theorem~5.3]{ArendtChernoffKato1982} or \cite[Corollary~2.3]{Arendt1987}. 
		
		The result is a special case of Theorem~\ref{thm:s_A_equals_omega_A} 
		since the existence of an interior point of $X^+$ implies that every norm bounded set in $X$ is order bounded, 
		see Lemma~\ref{lem:properties-interior-points} below.
		We will give another proof of the equality $\spb(A) = \omega(A)$ for $X^+$ with non-empty interior
		in Corollary~\ref{cor:equality-spectral-and-growth-bounds} below.
		
		Typical examples of spaces $X$ in which the cone has non-empty interior 
		are spaces $C(K)$ of continuous functions on compact Hausdorff spaces $K$ 
		and the self-adjoint parts of unital $C^*$-algebras.
		
		\item 
		Every compact set in $X$ is order bounded.
		
		This is equivalent to so-called \emph{$\alpha$-directedness} of $X$, 
		see \cite[Theorem~1]{Wickstead1975}, and the equality $\spb(A) = \omega(A)$ on such spaces 
		was proved in \cite[Theorem~4]{BattyDavies1983} and \cite[Corollary~2.4.5]{BattyRobinson1984}. 
		The equality is also an immediate consequence of Theorem~\ref{thm:s_A_equals_omega_A} 
		since the set $\{e^{tA}f: \, t \in [0,1]\}$ is compact and thus,
		due to the assumption on $X$, order bounded for each $f \in X$.
		
		Typical examples of such spaces $X$ are spaces $C_0(L)$ of continuous functions 
		that vanish at infinity on locally compact Hausdorff spaces $L$ 
		and the self-adjoint parts of arbitrary $C^*$-algebras.
		
		\item 
		There exists a time $t_0 \ge 0$ and a vector $h \in X_+$ such that the range $e^{t_0A}X$ is contained 
		in the so-called \emph{principal ideal} $X_h := \bigcup_{c \in [0,\infty)} [-ch,ch]$.
		
		In the special case where $X$ is a Banach lattice, 
		a proof of $\spb(A) = \omega(A)$ in this case can be found in \cite[Theorem~C-IV-1.1(b) on p.\,334]{Nagel1986}. 
		In the general case, the equality follows from Theorem~\ref{thm:s_A_equals_omega_A} by the following argument: 
		
		Since $X^+$ is normal, one can show that $X_h$ is a Banach space with respect to the so-called 
		\emph{gauge} norm $\norm{\argument}_h$ given by $\norm{x}_h := \min \{c \ge 0: \, x \in [-ch,ch]\}$ for all $x \in X_h$. 
		Moreover, also by the normality of the cone, $X_h$ endowed with this norm embeds continuously into $X$. 
		It thus follows from the closed graph theorem that $e^{t_0A}$ is a continuous operator from $X$ to $X_h$. 
		So we conclude from the semigroup law that, for each $f \in X$, the set $\{e^{tA}f: \, t \in [t_0, t_0+1]\}$ is norm bounded in $X_h$ 
		and hence order bounded in $X$.
	\end{enumerate}
\end{examples}

We point out that the equality $\spb(A) = \omega(A)$ for generators of positive semigroups 
does not hold on general ordered Banach spaces with normal and generating cone; 
see for instance \cite[Example~5.1.11]{ArendtBattyHieberNeubrander2011} for a counterexample.

\section{Stability if the cone has interior points}
\label{section:interior-points}

In this section, we consider the case where the cone in an ordered Banach space $(X,X^+)$ has non-empty (topological) interior
and show a similar result as Theorem~\ref{thm:stability-for-pos-sg} for this case;
the point is that the additional assumption on the cone allows for characterizations of $\spb(A)$ 
by a priori weaker statements. 
Moreover, it gives us equality of the spectral and the growth bound of positive semigroups, 
as pointed out in Example~\ref{exas:s_A_equals_omega_A}\ref{exas:s_A_equals_omega_A:item:non-empty-interior} above.

Note that the assumption that the cone $X^+$ has non-empty interior is rather strong; 
it is, for instance, satisfied for spaces of continuous functions over compact sets, and for $L^\infty$-spaces, 
but it is not satisfied on $L^p$-spaces for $p < \infty$, unless the space is finite-dimensional.

We will need the following well-known equivalences. 
For a more detailed discussion we refer for instance to \cite[Definition~2.4(iii) and~(v) and Proposition~2.11]{GlueckWeber2020}. 

\begin{lemma}
	\label{lem:properties-interior-points} 
	Let $(X,X^+)$ be an ordered Banach space and let $z \in X^+$. The following statements are equivalent:
	\begin{enumerate}[label=\upshape(\roman*)]
		\item\label{lem:properties-interior-points:itm:interior-point} 
		The vector $z$ is an element of the topological interior of $X^+$.
		
		\item\label{lem:properties-interior-points:itm:order-unit} 
		The vector $z$ is an \emph{order unit}, i.e., for each $y \in X$ there exists $\varepsilon > 0$ such that $z \ge \varepsilon y$.
		
		\item\label{lem:properties-interior-points:itm:unit-ball} 
		There is $c>0$, such that for all $y \in X$ with $\|y\|\leq c$, we have $-z \le y \le z$.
	\end{enumerate}
\end{lemma}

For two vectors $x,y \in X$, we write $x \ll y$ (or $y \gg x$) 
if there exists an interior point $z$ of $X^+$ such that $x+z \le y$.

We will need the following result about the existence of dual eigenvectors:

\begin{proposition}
	\label{prop:dual-eigenvector}
	Let $(X,X^+)$ be an ordered Banach space and suppose that the cone $X^+$ is normal and has non-empty interior. 
	Let $A: X \supseteq \dom{A} \to X$ be a densely defined resolvent positive linear operator.
	
	Then the spectral bound $\spb(A)$ satisfies $\spb(A) > -\infty$, it is an eigenvalue of the dual operator $A'$, 
	and there exists a corresponding eigenvector $0 \le z' \in \dom{A'}$.
\end{proposition}

Our proof of the proposition is a simple adaptation of an analogous result for positive operators in \cite[Eigenvalue Theorem on p.\,705]{Schaefer1967}. We need the following simple observation: we can identify the square $X' \times X'$ with the dual of the Banach space $X \times X$ by identifying each pair $(x', y') \in X' \times X'$ with the functional
\begin{align*}
	X \times X \ni (x,y) \mapsto \langle x',x\rangle + \langle y', y\rangle \in \R.
\end{align*}
By means of this identification, $X' \times X'$ carries a weak${}^*$-topology, which is easily seen to coincide with the product topology of the weak${}^*$-topology on $X'$.

\begin{proof}[Proof of Proposition~\ref{prop:dual-eigenvector}]
	It was shown in \cite[bottom of page~174]{ArendtChernoffKato1982} that $\spec(A)$ is non-empty, 
	so $\spb(A) > -\infty$.
	As $A$ is densely defined, 
	the dual operator $A': X' \supseteq \dom{A'} \to X'$ is well-defined.
	Furthermore, $\sigma(A') = \sigma(A)$ and thus $\spb(A') = \spb(A)$  
	(see \cite[Theorem~2 on p.\,225]{Yosida1980}). 
	As for $\mu\in(\spb(A'),+\infty)$ we have that $\Res(\mu,A') = (\Res(\mu,A))'$ 
	\cite[Theorem~2 on p.\,225]{Yosida1980}, 
	and since the dual of a positive operator is again positive, 
	we obtain that $A'$ is resolvent positive on the ordered Banach space $(X', (X')^+)$, 
	with $\spb(A') = \spb(A)$.
	
	Since the dual cone $(X')^+$ is generating and normal, it follows from 
	Proposition~\ref{prop:properties-of-resolvent-positive-operators}%
	\ref{prop:properties-of-resolvent-positive-operators:itm:spectral-bound-equal-real-spectral-bound} 
	that $\spb(A')=\realspb(A')$.
	Hence, it follows from 
	Lemma~\ref{lem:positive-approximate-eigenvector} that $\spb(A')$ 
	is even an approximate eigenvalue of $A'$ 
	with an approximate eigenvector $(x_n')$ in $(X')^+$.
	
	As $\|x_n'\|=1$ for all $n\in\N$, by the Banach--Alaoglu theorem \cite[Theorem~V.3.1 on p.~130]{Con90} there is a 
	weak${}^*$-convergent subnet of $(x_n')$, whose
	weak${}^*$-limit we denote by $z' \in (X')^+$. 
	Since the graph of $A'$ is weak${}^*$-closed in $X' \times X'$ 
	\cite[Proposition~1.1.1]{vanNeerven1992}, 
	it follows that $z' \in \dom{A'}$ and $A'z' = \spb(A)z'$. 

	So it only remains to show that $z'$ is non-zero. 
	To this end, fix an interior point $z \in X^+$ of $X^+$. 
	Then there exists a number $\varepsilon > 0$ 
	such that the order interval $[-z,z]$ contains $B_\varepsilon:=\{x\in X:\|x\|\leq \varepsilon\}$; 
	see Lemma~\ref{lem:properties-interior-points}\ref{lem:properties-interior-points:itm:unit-ball}. 
	This implies that for each $x' \in (X')^+$ and each $x \in B_\varepsilon$
	\[
		 \langle x', -z \rangle \leq  \langle x', x \rangle \leq  \langle x', z \rangle, 
	\]
	and thus 
	\[
		\langle x', z \rangle \geq \sup_{x\in B_\varepsilon} |\langle x', x \rangle| = \varepsilon \norm{x'}.
	\]
	Hence, we have $\langle x_n', z \rangle \ge \varepsilon$ for each $n$ and thus, 
	$\langle z', z \rangle \ge \varepsilon$, so $z'$ is indeed non-zero.
\end{proof}

Resolvent positivity does not guarantee that the resolvent maps the interior of $X^+$ into itself. 
The next proposition gives a simple yet useful characterization of this property.

\begin{proposition}
	\label{prop:Resolvent-mapping-interior-into-interior} 
	Let $(X,X^+)$ be an ordered Banach space with $\intt(X^+) \neq \emptyset$. 
	Let $A: X \supseteq \dom{A} \to X$ be a resolvent positive linear operator and $\lambda>\realspb(A)$. 
	The following statements are equivalent:
	\begin{enumerate}[label=\upshape(\roman*)]
		\item\label{prop:Resolvent-mapping-interior-into-interior:itm:strict-positivity} $\Res(\lambda,A)$ maps $\intt(X^+)$ into $\intt(X^+)$. 
		\item\label{prop:Resolvent-mapping-interior-into-interior:itm:domain} $D(A) \cap\intt(X^+) \neq\emptyset$.
	\end{enumerate}
\end{proposition}

\begin{proof}
	\Implies{prop:Resolvent-mapping-interior-into-interior:itm:strict-positivity}{prop:Resolvent-mapping-interior-into-interior:itm:domain}
	This is clear, as $D(A) = \Res(\lambda,A)X$ and $\intt(X^+) \neq \emptyset$.

	\Implies{prop:Resolvent-mapping-interior-into-interior:itm:domain}{prop:Resolvent-mapping-interior-into-interior:itm:strict-positivity}
	Let  $z \in D(A) \cap \intt(X^+)$. 
	Then there is $x \in X$ such that  $\Res(\lambda,A)x = z$.
	Let $y$ be an arbitrary interior point of $X^+$. 
	By Lemma~\ref{lem:properties-interior-points}\ref{lem:properties-interior-points:itm:order-unit}, 
	there is $\varepsilon > 0$ with $y\ge \varepsilon x$.
	As $A$ is resolvent positive, by 
	Proposition~\ref{prop:properties-of-resolvent-positive-operators}\ref{prop:properties-of-resolvent-positive-operators:itm:positive-and-decreasing}, 
	$\Res(\lambda,A)$ is a positive operator, and thus 
	\[
		\Res(\lambda,A)y \ge \varepsilon \Res(\lambda,A)x = \varepsilon z \gg 0,
	\]
	which proves~\ref{prop:Resolvent-mapping-interior-into-interior:itm:strict-positivity}.
\end{proof}

Note that, if $X^+$ has non-empty interior, every densely defined operator $A$ (and thus, for instance, every generator of a $C_0$-semigroup) satisfies condition~\ref{prop:Resolvent-mapping-interior-into-interior:itm:domain} in Proposition~\ref{prop:Resolvent-mapping-interior-into-interior}.
On the other hand, it is easy to find operators that do not satisfy the equivalent conditions of the previous proposition (and are thus not densely defined). Here is an example:

\begin{example}
	\label{ex:no-interior-point-in-domain}
	Endow the space $\ell_\infty$ with the usual cone $\ell_\infty^+$ of sequences 
	that are nonnegative in each component.
	Clearly, $\ell_\infty^+$ has non-empty interior.
	The multiplication operator $A: \ell_\infty \supseteq \dom{A} \to \ell_\infty$ that is given by
	\begin{align*}
		\dom{A} & = \left\{ x = (x_n)_{n \in \N} \in \ell_\infty: \; (-nx_n)_{n \in \N} \in \ell_\infty \right\}, \\
		Ax      & = (-nx_n)_{n \in \N} 
	\end{align*}
	has spectral bound $\spb(A) = -1$ and is resolvent positive, 
	but $\dom{A}$ does not contain any interior points of $\ell_\infty^+$. 
\end{example}

We note in passing that, as the operator $A$ in the previous example is not densely defined, it does not generate a $C_0$-semigroup.
(And in fact, all $C_0$-semigroups on $\ell^\infty$ have bounded generator \cite[Theorem~A-II-3.6(2)]{Nagel1986}.)

We will exploit also the following result, which is closely related to 
\cite[formula~(5.2) in Theorem~5.3]{ArendtChernoffKato1982}
and \cite[Proposition 3.9]{GlM21}:

\begin{proposition}
	\label{prop:spr-super-eigenvector-res-positive}
	Let $(X,X^+)$ be an ordered Banach space with $\intt\, X^+ \neq \emptyset$ 
	and let $A: X \supseteq \dom{A} \to X$ be a densely defined resolvent positive linear operator. 
	Then 
	\begin{align*}
		\realspb(A) \geq 
		\inf \left\{ \lambda \in [0,\infty) : \; \exists x \in \intt(X^+)\cap\dom{A} \text{ s.t. } Ax \ll \lambda x \right\}.
	\end{align*}
\end{proposition}

\begin{proof}
	Let $\lambda > \realspb(A)$. 
	Then $\lambda$ is in the resolvent set of $A$ and, 
	as $A$ is a resolvent positive operator, 
	Proposition~\ref{prop:properties-of-resolvent-positive-operators}\ref{prop:properties-of-resolvent-positive-operators:itm:positive-and-decreasing}
	implies that $\Res(\lambda,A)$ is a positive operator.
	
	Take an arbitrary point $y \in \intt(X^+)$ and define $x:= \Res(\lambda,A)y \in \dom{A}^+$.
	As $A$ is densely defined, clearly $D(A)\cap \intt(X^+)\neq\emptyset$, 
	so by Proposition~\ref{prop:Resolvent-mapping-interior-into-interior}, $x \in \dom{A} \cap \intt(X^+)$.
	Now, we estimate $Ax$ as
	\begin{align*}
		Ax
		& = \big(A - \lambda \id + \lambda \id\big) \, \Res(\lambda,A) y 
		 = - y + \lambda\Res(\lambda,A) y  = -y + \lambda x \ll \lambda x.
	\end{align*}
	This completes the proof.
\end{proof}

\begin{remark}
	\label{rem:ArendtChernoffKato1982} 
	If the cone $X^+$ is in addition normal, then a stronger counterpart of Proposition~\ref{prop:spr-super-eigenvector-res-positive} holds, see \cite[Theorem~5.3]{ArendtChernoffKato1982}.
\end{remark}

The following theorem, which characterizes the stability of resolvent positive operators in case that $X^+$ has non-empty interior, complements \cite[Theorem~3.10]{GlM21}, where powers of positive operators are considered.

\begin{theorem}
	\label{thm:stability-for-pos-sg-interior-point}
	Let $(X,X^+)$ be an ordered Banach space and suppose that the cone $X^+$ is normal and has non-empty interior. 
	Let $A: X \supseteq \dom{A} \to X$ be densely defined and resolvent positive linear operator.
	
	Then $A$ generates a positive $C_0$-semigroup on $X$, 
	and the following assertions are equivalent:
	\begin{enumerate}[label=\upshape(\roman*)]
		\item\label{thm:stability-for-pos-sg-interior-point:item:stability} 
		\emph{Spectral stability:}
		The spectral bound of $A$ satisfies $\spb(A) < 0$.
		
		\item\label{thm:stability-for-pos-sg-interior-point:item:dual-small-gain} 
		\emph{Dual small-gain condition:}
		For each non-zero $0 \le x' \in \dom{A'}$ we have
		\begin{align*}
			A'x' \not\ge 0.
		\end{align*}
		
		\item\label{thm:stability-for-pos-sg-interior-point:item:small-gain-interior-point-all} 
		\emph{Interior point small-gain condition, first version:}
		For every interior point $z$ of $X^+$ there is a number $\eta > 0$ such that
		\begin{align}
		\label{eq:interior-point-ver1}
			Ax \not\ge -\eta \norm{x} z \qquad \text{for all non-zero} \quad  x \in \dom{A}^+.
		\end{align}
		
		\item\label{thm:stability-for-pos-sg-interior-point:item:small-gain-interior-point-exists} 
		\emph{Interior point small-gain condition, second version:}
		There exists an interior point $z$ of $X^+$ and a number $\eta > 0$ such that
		\begin{align}
		\label{eq:interior-point-ver2}
			Ax \not\ge -\eta \norm{x} z \qquad \text{for all non-zero} \quad x \in \dom{A}^+.
		\end{align}
		
		\item\label{thm:stability-for-pos-sg-interior-point:item:strictly-decreasing-1}
		\emph{Strong decreasing property, first version:}
		There exists an interior point $z$ of $X^+$ belonging to $\dom{A}$ 
		such that $Az \ll 0$.

		\item\label{thm:stability-for-pos-sg-interior-point:item:strictly-decreasing-3}
		\emph{Strong decreasing property, second version:}
		There exists an interior point $z$ of $X^+$ belonging to $\dom{A}$ 
		and a number $\lambda < 0$ such that $Az \le \lambda z$.
		
		\item\label{thm:stability-for-pos-sg-interior-point:item:strongly-stable} 
		\emph{Strong stability:}
		For each $x \in X$ we have $e^{tA}x \to 0$ as $t \to \infty$.
		
		\item\label{thm:stability-for-pos-sg-interior-point:item:weakly-attractive}
		\emph{Weak attractivity on the cone:}
		For each $x \in X^+$ we have $\inf_{t \ge 0} \norm{e^{tA}x} = 0$.
		
		\item\label{thm:stability-for-pos-sg-interior-point:item:uniformly-exp-stable}
		\emph{Uniform exponential stability:}
		The growth bound of the semigroup $(e^{tA})_{t \ge 0}$ satisfies $\omega(A) < 0$.
	\end{enumerate}	
\end{theorem}

\begin{proof}
	The fact that $A$ generates a positive $C_0$-semigroup
	is shown in \cite[Theorem~5.3]{ArendtChernoffKato1982}.
	Let us now prove that claimed equivalences.
	
	\Implies{thm:stability-for-pos-sg-interior-point:item:stability}{thm:stability-for-pos-sg-interior-point:item:dual-small-gain}
	Since $\spb(A) = \spb(A') < \infty$, it follows from Proposition~\ref{prop:properties-of-resolvent-positive-operators}\ref{prop:properties-of-resolvent-positive-operators:itm:positive-and-decreasing}
	that $\Res(0,A') = \Res(0,A)' \ge 0$. Assume now that $0 \le x' \in \dom{A'}$ and that $A'x' \ge 0$. By applying the positive resolvent $\Res(0,A')$ to this inequality, we obtain $-x' \ge 0$, so $x' \le 0$. Since $x'$ was assumed to be positive, it follows that $x' = 0$.
	
	\Implies{thm:stability-for-pos-sg-interior-point:item:dual-small-gain}{thm:stability-for-pos-sg-interior-point:item:stability}
	Assume that the number $\spb(A) = \spb(A')$ is at least $0$.
	According to Proposition~\ref{prop:dual-eigenvector}, $\spb(A')$ is an eigenvalue of $A'$ with an eigenvector $0 \le z' \in \dom{A'}$. Hence,
	\begin{align*}
		A' z' = \spb(A') z' \ge 0,
	\end{align*}
	which contradicts~\ref{thm:stability-for-pos-sg-interior-point:item:dual-small-gain} since $z'$ is positive and non-zero.
	
	\Implies{thm:stability-for-pos-sg-interior-point:item:stability}%
	{thm:stability-for-pos-sg-interior-point:item:small-gain-interior-point-all} 
	In view of Theorem~\ref{thm:stability-for-pos-sg}\ref{thm:stability-for-pos-sg:itm:uniform-small-gain},
	the condition $\spb(A)<0$ implies the uniform small-gain condition, i.e., there is $\tilde \eta>0$ such that 
	\begin{align}
		\label{eq:Uniform-SGC-tmp}
		\dist(Ax,X^+) \ge \tilde \eta \norm{x} 
		\qquad 
		\text{for all } x \in \dom{A}^+.
	\end{align}
	Let $z$ be an interior point of $X^+$. 
	We show that~\eqref{eq:interior-point-ver1} holds with $\eta := \frac{\tilde \eta}{2\|z\|}$. 
	Indeed, suppose that~\eqref{eq:interior-point-ver1} fails for this $\eta$.
	Then there exists a non-zero vector $x \in \dom{A}^+$ such that 
	$Ax  + \eta \norm{x} z \geq 0$.
	Hence 
	\[
		\dist(Ax,X^+) 
		\leq 
		\norm{ 
			Ax - \big(Ax  + \eta \norm{x} z\big) 
		}  
		= 
		\frac{\tilde \eta}{2} \norm{x}.
	\]
	This together with~\eqref{eq:Uniform-SGC-tmp} can only hold if $x = 0$,
	so we arrived at a contradiction.

	\Implies{thm:stability-for-pos-sg-interior-point:item:small-gain-interior-point-all}{thm:stability-for-pos-sg-interior-point:item:small-gain-interior-point-exists}
	This implication is obvious.
	
	\Implies{thm:stability-for-pos-sg-interior-point:item:small-gain-interior-point-exists}{thm:stability-for-pos-sg-interior-point:item:stability}
	In view of Theorem~\ref{thm:stability-for-pos-sg}\ref{thm:stability-for-pos-sg:itm:uniform-small-gain}, it suffices to show that the uniform small-gain condition holds. 
	So let $z\in\intt(X^+)$ and $\eta > 0$ be as in~\ref{thm:stability-for-pos-sg-interior-point:item:small-gain-interior-point-exists}.
	For the interior point $z$ of $X^+$ pick a corresponding number $c>0$ 
	as in Lemma~\ref{lem:properties-interior-points}\ref{lem:properties-interior-points:itm:unit-ball}.
	We show that the uniform small-gain condition \eqref{eq:Uniform-SGC-tmp} holds with $\tilde \eta := c\eta$.
	Indeed, suppose that this is not the case.
	Then there is $x \in \dom{A}^+$ such that 
	\begin{align*}
		\dist(Ax,X^+) < c \eta \norm{x}.
	\end{align*}
	Thus, $x \not= 0$ and there is $y \in X^+$ such that $  \frac{\left\| Ax - y \right\|}{\eta\|x\|} \le c$. 
	By the choice of $c$ we obtain 
	\[
		Ax \geq Ax - y \ge - \eta \|x\| z, 
	\] 
	which contradicts~\ref{thm:stability-for-pos-sg-interior-point:item:small-gain-interior-point-exists}.
	
	\Implies{thm:stability-for-pos-sg-interior-point:item:stability}{thm:stability-for-pos-sg-interior-point:item:strictly-decreasing-1} 
	This implication follows by Proposition~\ref{prop:spr-super-eigenvector-res-positive}.
	
	\Implies{thm:stability-for-pos-sg-interior-point:item:strictly-decreasing-1}{thm:stability-for-pos-sg-interior-point:item:strictly-decreasing-3} Take $z$ as in item \ref{thm:stability-for-pos-sg-interior-point:item:strictly-decreasing-1}. 
	As $-Az \in\intt(X^+)$, by Lemma~\ref{lem:properties-interior-points} 
	there is $\lambda<0$ such that $-Az \geq -\lambda z$, which implies that $Az \leq \lambda z$.
	
	\Implies{thm:stability-for-pos-sg-interior-point:item:strictly-decreasing-3}{thm:stability-for-pos-sg-interior-point:item:strongly-stable}
	Let $z$ be an interior point of $X^+$ that is an element of $\dom{A}$ 
	and assume that $\lambda z \ge A z$ for a real number $\lambda < 0$. 
	From this estimate, we first derive a corresponding estimate for the resolvent of $A$, 
	and then -- by means of Euler's formula -- an estimate for the semigroup generated by $A$:
	
	For each real number $\mu > \max\{\spb(A),0\}$ the resolvent $\Res(\mu,A)$ is positive
	by Proposition~\ref{prop:properties-of-resolvent-positive-operators}%
	\ref{prop:properties-of-resolvent-positive-operators:itm:positive-and-decreasing}, so we obtain
	\begin{align*}
		\lambda \Res(\mu,A)z \ge \Res(\mu,A)Az = -z + \mu \Res(\mu,A)z,
	\end{align*}
	which implies that
	\begin{align*}
		\frac{\mu}{\mu-\lambda} z \ge \mu \Res(\mu,A)z,
		\qquad \text{and hence} \qquad
		\left(\frac{\mu}{\mu-\lambda}\right)^n z \ge \Big(\mu \Res(\mu,A)\Big)^n z
	\end{align*}
	for each $n \in \N_0$. 
	Now, fix a time $t > 0$. 
	If $n \in \N$ is such large that $n/t > \spb(A)$, 
	then the preceding inequality yields, 
	by substituting $\mu = n/t$,
	\begin{align*}
		\left( 1 - \frac{t\lambda}{n} \right)^{-n} z \ge \Big(\frac{n}{t} \Res\big(\frac{n}{t},A\big)\Big)^n z.
	\end{align*}
	The left hand side converges to $e^{t\lambda}z$ as $n \to \infty$, 
	and the right-hand side converges to $e^{tA}z$ as $n \to \infty$ 
	by Euler's formula for $C_0$-semigroups \cite[Corollary~III.5.5]{EngelNagel2000}.
	
	Hence, $e^{tA}z \le e^{t\lambda}z$ for each $t > 0$. 
	Since $\lambda < 0$, this implies that $e^{tA} z \to 0$ as $t \to \infty$. 
	
	As $z$ is an interior point of $X^+$, 
	there is a non-zero multiple of the unit ball in $X$ 
	which is contained in the order interval $[-z,z]$, see Lemma~\ref{lem:properties-interior-points}. 
	Tshe normality of the cone thus implies 
	that we even have $e^{tA} x \to 0$ as $t \to \infty$ for each $x$ from the unit ball, 
	and by linearity of the semigroup, the same holds for all $x \in X$.
	
	\Implies{thm:stability-for-pos-sg-interior-point:item:strongly-stable}{thm:stability-for-pos-sg-interior-point:item:weakly-attractive}
	This implication is obvious.
	
	\Implies{thm:stability-for-pos-sg-interior-point:item:weakly-attractive}{thm:stability-for-pos-sg-interior-point:item:uniformly-exp-stable}
	Let $z$ be an interior point of $X^+$. 
	By multiplying $z$ with a positive scalar if necessary, 
	we may assume that the unit ball in $X$ is contained in the order interval $[-z,z]$, 
	see Lemma~\ref{lem:properties-interior-points}. 
	Moreover, due to the normality of the cone, 
	there exists a number $C > 0$ for which the inequality~\eqref{eq:normality-estimate} holds. 
	Now, choose a time $t_0 > 0$ such that $\norm{e^{t_0 A} z} \le \frac{1}{2C}$. 
	For each $x$ in the unit ball of $X$ we have by positivity of the semigroup 
	that $e^{t_0 A}x \in [-e^{t_0 A}z, e^{t_0 A}z]$, so
	\begin{align*}
		\norm{e^{t_0 A}x} \le C \norm{e^{t_0 A} z} \le \frac{1}{2}.
	\end{align*}
	Hence, the operator $e^{t_0 A}$ has norm at most $\frac{1}{2}$, 
	which proves that the $n$-th powers of $e^{t_0 A}$ 
	converge to $0$ with respect to the operator norm as $n \to \infty$. 
	Since the $C_0$-semigroup $(e^{tA})_{t \ge 0}$ is operator norm bounded 
	on the compact time interval $[0,t_0]$, 
	we thus obtain $e^{tA} \to 0$ with respect to the operator norm as $t \to \infty$.
	
	\Implies{thm:stability-for-pos-sg-interior-point:item:uniformly-exp-stable}{thm:stability-for-pos-sg-interior-point:item:stability}
	This implication is a consequence of the general fact 
	that the spectral bound of a semigroup generator 
	is dominated by the growth bound of the semigroup \cite[Corollary~II.1.13]{EngelNagel2000}.
\end{proof}

The fact, mentioned in the theorem, 
that a resolvent positive and densely defined linear operator
generates a positive $C_0$-semigroup, 
is a consequence of the normality of the cone and of the existence of an interior point of the positive cone
\cite[Theorem~5.3]{ArendtChernoffKato1982};
this is not true for more general ordered Banach spaces. 

The equivalent condition~\ref{thm:stability-for-pos-sg-interior-point:item:dual-small-gain} 
in Theorem~\ref{thm:stability-for-pos-sg-interior-point} is particularly nice 
since it does not require any kind of uniform estimate 
(in contrast to condition~\ref{thm:stability-for-pos-sg:itm:uniform-small-gain} 
in Theorem~\ref{thm:stability-for-pos-sg}).
The following simple example, which is an adaptation of Example~\ref{exa:shift-non-uniform}\ref{exa:shift-non-uniform:item:differential} above, 
shows that the equivalence of~\ref{thm:stability-for-pos-sg-interior-point:item:stability}
and~\ref{thm:stability-for-pos-sg-interior-point:item:dual-small-gain} 
in Theorem~\ref{thm:stability-for-pos-sg-interior-point} does not hold in general, if $X^+$ has empty interior.

\begin{example}
	\label{exa:shift-empty-interior} 
	Let $(e^{tA})_{t \ge 0}$ denote the right shift semigroup on $X = L^2(\R)$, i.e.\ 
	$e^{tA}f = f(\argument-t)$ for all $f \in L^2(\R)$ and all $t \ge 0$. 
	Here, we endow $L^2(\R)$ with the pointwise almost everywhere order.
	The domain of the generator $A$ is the Sobolev space $H^1(\R)$ and $A$ acts $Af = -f'$ for all $f \in H^1(\R)$. 
	The dual operator $A'$ also has domain $H^1(\R)$ and one has $A'f = f'$ for each $f \in H^1(\R)$. 
	
	So if $0 \le f \in H^1(\R)$ satisfies $f' = A'f \ge 0$, 
	then $f$ is increasing and hence $f=0$ since $f \in L^2(\R)$. 
	This shows that condition~\ref{thm:stability-for-pos-sg-interior-point:item:dual-small-gain} 
	in Theorem~\ref{thm:stability-for-pos-sg-interior-point} is satisfied. 
	However, one has $\spb(A) = 0$ since one can see by a Fourier transform argument that $\sigma(A) = i\R$.
\end{example}

As an immediate consequence of Theorem~\ref{thm:stability-for-pos-sg-interior-point}, 
one re-obtains the following classical result which we already discussed in 
Example~\ref{exas:s_A_equals_omega_A}\ref{exas:s_A_equals_omega_A:item:non-empty-interior} 
since it can also be derived as a consequence of Theorem~\ref{thm:s_A_equals_omega_A}.

\begin{corollary}
	\label{cor:equality-spectral-and-growth-bounds} 
	Let the assumptions of Theorem~\ref{thm:stability-for-pos-sg-interior-point} hold.
	Then $\spb(A)=\omega(A)$.
\end{corollary}

\begin{proof}
	Assume for a contradiction that $\spb(A) < \omega(A)$. 
	Then there exists a number $c \in \R$ such that
	\begin{align*}
		\spb(A-c\id) < 0 < \omega(A-c\id).
	\end{align*}
	This contradicts the equivalence of
	items~\ref{thm:stability-for-pos-sg-interior-point:item:stability}
	and~\ref{thm:stability-for-pos-sg-interior-point:item:uniformly-exp-stable} of 
	Theorem~\ref{thm:stability-for-pos-sg-interior-point}
	since the operator $A-c\id$ satisfies the assumptions of the theorem.
\end{proof}

\section{A Krein--Rutman type theorem and its consequences for stability}
\label{section:quasi-compact}

In this section, we consider resolvent positive operators $A$
for which the essential spectral bound is strictly negative;
for such operators we first discuss an analogue of the Krein--Rutman theorem,
and then give a version of the stability result in Theorem~\ref{thm:stability-for-pos-sg};
the advantages if $A$ has strictly negative essential spectral bound, 
are that we obtain a simpler characterization of (spectral) stability
and that, at the same time, we need fewer assumptions on the underlying space.

\subsection{The essential spectrum}
\label{subsection:quasi-compact:the-essential-spectrum}

Let us first recall a few facts about the essential spectrum of unbounded operators.
A bounded linear operator $T$ between two Banach spaces $W$ and $X$ is called a \emph{Fredholm operator} 
if its kernel has finite dimension and its range has finite co-dimension
(and as a consequence of the latter property, the range of $T$ is then automatically closed;
see for instance \cite[Corollary~XI.2.3 on p.\,187]{GohbergGoldbergKaashoek1990}).
Obviously, if $T$ is bijective, then it is Fredholm.
One can prove that the set of all Fredholm operators from $W$ to $X$ is open 
(\cite[Theorem~XI.4.1 on p.\,189]{GohbergGoldbergKaashoek1990}).

Now, let $A: X \supseteq \dom{A} \to X$ be a closed linear operator on a complex Banach space $X$.
Then $\dom{A}$ is a Banach space with respect to the graph norm, 
and a complex number $\lambda$ is said to be in the \emph{essential spectrum} $\specEss(A)$ if
the mapping $\lambda \id - A$ from the Banach space $\dom{A}$ to the Banach space $X$ is not Fredholm
(here, $\id: \dom{A} \to X$ denotes the canonical injection). 
Let us recall a few standard facts about the essential spectrum in the following proposition.

\begin{proposition}
	\label{prop:Essential-spectrum-simple-properties} 
	Let $A: X \supseteq \dom{A} \to X$ be a closed linear operator on a complex Banach space $X$. 
	\begin{enumerate}[label=\upshape(\alph*)]
		\item\label{prop:Essential-spectrum-simple-properties:itm:closed-and-spectral} 
		The essential spectrum of $A$ is closed and it is contained in the spectrum of $A$. 
		
		\item\label{prop:Essential-spectrum-simple-properties:itm:compact-resolvent} 
		If $A$ has non-empty resolvent set and compact resolvent, 
		then $\specEss(A) = \emptyset$.
	\end{enumerate}
\end{proposition}

\begin{proof}
	\ref{prop:Essential-spectrum-simple-properties:itm:closed-and-spectral}
	The closedness of $\specEss(A)$ follows from the fact that the set of Fredholm operators is open, 
	and the inclusion $\specEss(A) \subseteq \spec(A)$ follows from the fact 
	that every bijective operator is Fredholm.
	
	\ref{prop:Essential-spectrum-simple-properties:itm:compact-resolvent} 
	Fix $\mu \in \resSet(A)$. 
	Then $\mu \id - A$ is bijective from $\dom{A}$ to $X$ and thus a Fredholm operator.
	Moreover, $\id: \dom{A} \to X$ is compact since $A$ has compact resolvent and thus, 
	for every $\lambda \in \C$ the operator
	\begin{align*}
		\lambda \id - A = (\lambda - \mu) \id + (\mu \id - A): \dom{A} \to X
	\end{align*}
	is a compact perturbation of a Fredholm operator and thus also Fredholm
	\cite[Theorem~XI.4.2 on p.\,189]{GohbergGoldbergKaashoek1990}.
\end{proof}

For a closed operator $A: X \supseteq \dom{A} \to X$ we call
\begin{align*}
	\spbEss(A) := \sup\{ \re \lambda: \; \lambda \in \specEss(A) \} \in [-\infty,\infty].
\end{align*}
the \emph{essential spectral bound} of $A$.
Assume that $\spbEss(A) < \infty$ and that the right half plain 
\begin{align*}
	\Omega := \{\lambda \in \C: \; \re \lambda > \spbEss(A) \}
\end{align*}
has non-empty intersection with the resolvent set of $A$. 
Then it follows from so-called \emph{analytic Fredholm theory}
that $A$ has at most countably many spectral values in $\Omega$, 
that all these spectral values are isolated in $\Omega$ 
(though some of them might accumulate at $\partial \Omega$), 
and that all spectral values of $A$ in $\Omega$ are poles of the resolvent of $A$ 
with finite-rank spectral projections.
Indeed, this follows immediately by applying \cite[Corollary~XI.8.4 on p.\,203]{GohbergGoldbergKaashoek1990}
to the operator mapping
\begin{align*}
	W: \Omega \to \calL(\dom{A}; X), 
	\qquad 
	\lambda \mapsto \lambda \id - A.
\end{align*}

\subsection{A Krein--Rutman type theorem}

Let us first recall the Krein--Rutman theorem for positive linear operators.
In one of its rather general versions,
it says the following: 
if $(X, X^+)$ is an ordered Banach space with total cone and $T: X \to X$ 
is a positive and bounded linear operator 
such that the essential spectral radius of $T$ satisfies
$\sprEss(T) < \spr(T)$, 
then $\spr(T)$ is an eigenvalue of $T$ and of the dual operator $T'$,
and both $T$ and $T'$ have a positive eigenvector for this eigenvalue.
The main difficulty in the proof is to show that $\spr(T)$ is in the spectrum of $T$;
this is quite easy if the cone is even assumed to be generating and normal, but it is more involved in the general case.

Still, a variety of different proofs is known for the theorem,
see for instance \cite[Theorem~6.1 on p.\,262]{Krein1950} for Krein and Rutman's classical proof 
based on a perturbation argument, 
\cite[Theorem~1]{ZabreikoSmickih1979} for an argument that reduces the theorem to the finite-dimensional case, 
\cite[Corollary~2.2 on p.\,324]{Nussbaum1981} for a proof based on non-linear arguments,
\cite[Theorem~6.1]{LiJia2021} for an argument based on the (long-time) behaviour of the powers $T^n$,
and \cite[2.4~on p.\,312]{SchaeferWolff1999} for a proof based on Pringsheim's theorem from complex analysis.

A version of the Krein-Rutman theorem for resolvent positive operators reads as follows:

\begin{theorem}[Krein--Rutman for resolvent positive operators]
	\label{thm:krein-rutman-resolvent-positive}
	Let $(X,X^+)$ be an ordered Banach space with total cone 
	and let $A: X \supseteq \dom{A} \to X$ be a resolvent positive linear operator.
	Assume that $\spbEss(A) < \spb(A)$.
	
	Then $\spb(A) < \infty$, 
	the number $\spb(A)$ is an eigenvalue of $A$ and 
	there exists a corresponding eigenvector in $X^+$.
	If $A$ is, in addition, densely defined, 
	then $\spb(A)$ is also an eigenvalue of the dual operator $A'$
	and there exists a corresponding eigenvector in $(X')^+$.
\end{theorem}

This theorem can essentially be found in \cite[Corollary~2.9(ii)]{KanigowskiKryszewski2012},
where it was shown by applying the Krein--Rutman theorem for positive operators
to the resolvent of $A$ and invoking a spectral mapping theorem; 
however, the property $\spb(A) < \infty$ is assumed (rather than proved) there, 
and the existence of a positive dual eigenvector is not mentioned.
A very similar result can also be found 
-- though only under the additional assumption that the cone be generating and normal --
in \cite[Proposition~3.10]{Thieme1998}.

It appears worthwhile to give an alternative and more direct proof of 
Theorem~\ref{thm:krein-rutman-resolvent-positive},
which does not rely on the Krein--Rutman theorem for positive operators.
We present such a proof below;
it is based on Bernstein's representation theorem for completely monotone functions.
This argument is close in spirit to the aforementioned proof of the Krein-Rutman theorem
that relies on Pringsheim's theorem.
It is also loosely reminiscent of the proof of \cite[Theorem~4.3]{Mui2023}. 

\begin{proof}[Proof of Theorem~\ref{thm:krein-rutman-resolvent-positive}]
	The first part of the proof is to show that $\spb(A) <\infty$ and $\spb(A) \in \spec(A)$.
	Afterwards, we will derive the remaining assertions from standard spectral theory.
	
	So assume to the contrary that either $\spb(A) = \infty$ 
	or that $\spb(A)$ is $< \infty$ but not in the spectrum.
	As $\realspb(A) \in [-\infty, \infty)$ is either $-\infty$ or in the spectrum of $A$, this means that $\realspb(A)<\spb(A)$.
	Thus, after a translation of $A$ by a real number we may, and will,
	assume that $\spbEss(A) < 0 < \spb(A)$, but that the set $[0,\infty)$ is in the resolvent set of $A$ (i.e., $\realspb(A)<0$).

	So we find a spectral value $\lambda_0$ of $A$ whose real part is strictly positive;
	since $\spbEss(A) < 0$, we know 
	from the properties listed in Subsection~\ref{subsection:quasi-compact:the-essential-spectrum}
	that $\lambda_0$ is an isolated spectral value of $A$  
	and a pole of the resolvent $\Res(\argument,A)$;
	let $k \ge 1$ denote pole order.
	Then the limit $Q_{-k} := \lim_{\lambda \to \lambda_0} (\lambda - \lambda_0)^k \Res(\lambda,A)$ 
	exists with respect to the operator norm,
	and is equal to the $-k$-th coefficient of the Laurent series expansion 
	of the resolvent about $\lambda_0$;
	in particular, $Q_{-k} \not= 0$.
	
	Now fix a vector $0 \le x \in X$ and a functional $0 \le x' \in X'$.
	We will show next that $\langle x', Q_{-k} x\rangle = 0$;
	this yields a contradiction to $Q_{-k} \not= 0$ since the span of $X^+$ is dense in $X$
	and the span of $(X')^+$ is weak${}^*$-dense in $X'$. 
	
	As explained in Subsection~\ref{subsection:quasi-compact:the-essential-spectrum},
	$A$ has at most countably many spectral values with strictly positive real part, 
	and all these spectral values are isolated. 
	Hence, the open set $D := \{\lambda \in \resSet(A): \, \re \lambda > 0\}$ is connected.
	Now consider the holomorphic mapping
	\begin{align*}
		f: D       & \to     \C, \\
		   \lambda & \mapsto \langle x', \Res(\lambda,A) x \rangle.
	\end{align*}
	The restriction of the mapping $f$ to the interval $(0,\infty)$
	is \emph{completely monotone},
	i.e., its derivatives satisfy $(-1)^n f^{(n)}(\lambda) \ge 0$ 
	for all $n \in \N_0$ and all $\lambda \in (0,\infty)$.
	Indeed, for $\lambda \in (0,\infty)$ we have
	\begin{align*}
		(-1)^n f^{(n)}(\lambda)
		= 
		\langle x', (-1)^n \Res^{(n)}(\lambda,A)x \rangle
		=
		\langle x', n! \Res(\lambda,A)^{n+1} x \rangle,
	\end{align*}
	and the operator $\Res(\lambda,A)$ is positive as shown 
	in Proposition~\ref{prop:properties-of-resolvent-positive-operators}\ref{prop:properties-of-resolvent-positive-operators:itm:positive-and-decreasing}.
	
	Since $f|_{(0,\infty)}$ is completely monotone,
	we can apply Bernstein's representation theorem for completely monotone functions
	(see e.g.\ \cite[Theorem~1.4]{SchillingSongVondracek2012}),
	which tells us that $f|_{(0,\infty)}$ 
	is the Laplace transform of a positive measure on $[0,\infty)$ 
	-- more precisely, there exists a positive (and $\sigma$-finite)
	measure $\mu$ on the Borel $\sigma$-algebra on $[0,\infty)$
	such that 
	\begin{align*}
		f(\lambda) = \int_{[0,\infty)} e^{-\lambda t} \dx\mu(t)
	\end{align*}
	for each $\lambda \in (0,\infty)$.
	This readily implies that the integral
	\begin{align*}
		g(\lambda) := \int_{[0,\infty)} e^{-\lambda t} \dx\mu(t)
	\end{align*}
	even converges for every complex number $\lambda$ with real part $> 0$,
	and that $g$ is an analytic function on the right half plane in $\C$.
	
	By the identity theorem for analytic functions,
	$f$ coincides with $g$ on $D$ (due to the connectedness auf $D$)
	and thus, in particular, in a pointed neighbourhood of $\lambda_0$. 
	Therefore,
	\begin{align*}
		\langle x', Q_{-k} x \rangle
		=
		\lim_{\lambda \to \lambda_0} (\lambda - \lambda_0)^k f(\lambda)
		= 
		\lim_{\lambda \to \lambda_0} (\lambda - \lambda_0)^k g(\lambda)
		=
		0,
	\end{align*}
	where the last equality follows from $k \ge 1$ and from the fact that
	$g$ is analytic in $\lambda_0$.
	So we arrived at our desired contradiction and have thus proved that $\spb(A)$ is finite 
	and a spectral value of $A$.
	
	The rest of the proof is now standard spectral theory:
	
	Since $\spbEss(A) < \spb(A)$, the number $\spb(A)$ is a pole of the resolvent
	and hence an eigenvalue of $A$.
	Let $j \ge 1$ denote its pole order and $R_{-j}$ the $-j$-th coefficient
	of the Laurent series expansion of $\Res(\argument,A)$.
	Then $R_{-j}$ is non-zero and all vectors in its range, except for $0$,
	are eigenvectors of $A$ for the eigenvalue $\spb(A)$.
	Since
	\begin{align*}
		R_{-j} 
		= 
		\lim_{\lambda \to \spb(A)} (\lambda - \spb(A))^{j} \Res(\lambda,A)
		= 
		\lim_{\lambda \downarrow \spb(A)} (\lambda - \spb(A))^{j} \Res(\lambda,A),
	\end{align*}
	were the limits exist in operator norm, 
	we conclude from the positivity of the resolvent on the right of $\spb(A)$
	(Proposition~\ref{prop:properties-of-resolvent-positive-operators}\ref{prop:properties-of-resolvent-positive-operators:itm:positive-and-decreasing})
	that the operator $R_{-j}$ is positive.
	
	The span of $X^+$ is dense in $X$ and $R_{-j}$ is non-zero, 
	there exists a vector $x \in X^+$ such $R_{-j}x \not= 0$.
	Hence, $R_{-j}x$ is a positive eigenvector of $A$ for the eigenvalue $\spb(A)$.
	
	Finally, assume in addition that $A$ is densely defined, 
	so that the dual operator $A'$ is well-defined.
	Then $A'$ has the same spectrum as $A$, 
	and $\Res(\lambda,A') = \Res(\lambda,A)'$ for each $\lambda$ in the resolvent set of $A$.
	Hence, $\spb(A)$ is also a pole of the resolvent of $A'$,
	and by repeating the previous argument
	(where we use now that the span of $(X')^+$ if weak${}^*$-dense in $X'$) 
	we also obtain a positive eigenvector of $A'$ for the eigenvalue $\spb(A)$.
\end{proof}

In the next subsections we show general results on how the Krein--Rutman type Theorem~\ref{thm:krein-rutman-resolvent-positive} 
can be used to derive information about the location of the spectral bound.
Before this, we find it worthwhile to demonstrate the use of Theorem~\ref{thm:krein-rutman-resolvent-positive} 
in a simple toy example.

\begin{example}
	Endow the Banach space $C_{\operatorname{per}}([0,1])$ of continuous functions $f$ on $[0,1]$ that satisfy $f(0) = f(1)$ 
	with the pointwise order. 
	Let $(e^{tA})_{t \in [0,\infty)}$ be the periodic right shift semigroup on $X$. 
	Its generator is given by
	\begin{align*}
		\dom{A} & = \{f \in X :\, f \text{ is differentiable and } f' \in X\}, \\
		Af & = -f'.
	\end{align*}
	Let $V: [0,1] \to \R$ be a continuous function that satisfies $V(0) = V(1)$ and 
	denote the bounded linear operator $X \to X$, $f \mapsto Vf$ also by $V$. 
	We will show that $\spb(A+V) = \int_0^1 V(x) \dx x$. 
	
	Note that this is relevant for the stability of the perturbed semigroup $(e^{t(A+V)})_{t \ge 0}$: 
	due to the choice of the space $X$, the growth bound $\omega()A+V$ of this semigroup 
	coincides with $\spb(A+V)$ (see Example~\ref{exas:s_A_equals_omega_A}\ref{exas:s_A_equals_omega_A:item:non-empty-interior}).
	Thus, the semigroup converges to $0$ in operator norm if and only if $\int_0^1 V(x) \dx x < 0$.
\end{example}

\begin{proof}
	First note that $A+V$ generates a positive semigroup 
	and that $A$, and thus $A+V$, has compact resolvent due to the Arzelà--Ascoli theorem. 
	Hence, $\specEss(A+V)$ is empty. 
	
	Set $s := \spb(A+V)$. 
	Since the cone $X_+$ is normal and has non-empty interior, 
	it follows from Proposition~\ref{prop:dual-eigenvector} 
	(or from \cite[Theorem~B-III.1.1]{Nagel1986}) 
	that $s > -\infty$. 
	Hence, $s$ is an eigenvalue of $A+V$ with an eigenvector $f \in X^+$ according to 
	Theorem~\ref{thm:krein-rutman-resolvent-positive}.
	So we have $sf = (A+V)f = -f' + Vf$. 
	Thus,  
	\begin{align*}
		f(x) = \exp\big(-sx + \int_0^x V(y) \dx y \big) f(0)
	\end{align*}
	for all $x \in [0,1]$. 
	Since $f \not= 0$ we conclude that $f(0) \not= 0$. 
	We now substitute $x = 1$ and use that $f(0) = f(1)$. 
	This gives $e^s = \exp\big(\int_0^1 V(y) \dx y\big)$ and hence $s = \int_0^1 V(y) \dx y$, as claimed.
\end{proof}

We note in passing that the proof did not use that fact the the eigenvector $f$ is in $X^+$. 
However, it used the fact that the spectral bound $s$ is indeed in the spectrum (and hence an eigenvalue).

\subsection{Spectral stability for operators with small essential spectrum}

As a consequence of the preceding Theorem~\ref{thm:krein-rutman-resolvent-positive},
we easily obtain the following characterization of spectral stability
for resolvent positive operators.

\begin{corollary}
	\label{cor:stability-for-pos-sg-quasi-compact}
	Let $(X,X^+)$ be an ordered Banach space with total cone and 
	let $A: X \supseteq \dom{A} \to X$ be a resolvent positive linear operator;
	assume that $\spbEss(A) < 0$. 
	Then the following assertions are equivalent:
	\begin{enumerate}[label=\upshape(\roman*)]
		\item\label{cor:stability-for-pos-sg-quasi-compact:itm:stability} 
		\emph{Spectral stability:}
		The spectral bound of $A$ satisfies $\spb(A) < 0$.
		
		\item\label{cor:stability-for-pos-sg-quasi-compact:itm:pos-resolvent}
		\emph{Positive resolvent at $0$:}
		The operator $A$ is invertible and the resolvent $-A^{-1} = \Res(0,A)$ is positive.
		
		\item\label{cor:stability-for-pos-sg-quasi-compact:itm:backwards-positive} 
		\emph{All $0$-super-eigenvectors of $A$ are negative:}
		If $x \in \dom{A}$ satisfies
		\begin{align*}
			Ax \ge 0
		\end{align*}
		then $x \le 0$.
		
		\item\label{cor:stability-for-pos-sg-quasi-compact:itm:small-gain} 
		\emph{Small-gain condition:}
		For each non-zero $0 \le x \in \dom{A}$ we have
		\begin{align*}
			Ax \not\ge 0.
		\end{align*}
	\end{enumerate}
\end{corollary}

\begin{proof}
	\Implies{cor:stability-for-pos-sg-quasi-compact:itm:stability}{cor:stability-for-pos-sg-quasi-compact:itm:pos-resolvent}
	This is a consequence of 
	Proposition~\ref{prop:properties-of-resolvent-positive-operators}\ref{prop:properties-of-resolvent-positive-operators:itm:positive-and-decreasing}.
	
	\Implies{cor:stability-for-pos-sg-quasi-compact:itm:pos-resolvent}{cor:stability-for-pos-sg-quasi-compact:itm:backwards-positive}
	Since $-A^{-1}$ is positive, it follows from $Ax \ge 0$ that $-x \ge 0$.
	
	\Implies{cor:stability-for-pos-sg-quasi-compact:itm:backwards-positive}{cor:stability-for-pos-sg-quasi-compact:itm:small-gain}
	Let $0 \le x \in \dom{A}$. 
	If $Ax \ge 0$, then it follows 
	from~\ref{cor:stability-for-pos-sg-quasi-compact:itm:backwards-positive}
	that $x \le 0$ and hence, $x = 0$.
	
	\Implies{cor:stability-for-pos-sg-quasi-compact:itm:small-gain}{cor:stability-for-pos-sg-quasi-compact:itm:stability}
	Assume that $\spb(A) \ge 0$.
	Since $\spbEss(A) < 0$, we can then apply Theorem~\ref{thm:krein-rutman-resolvent-positive}
	and conclude that $\spb(A)$ is finite and an eigenvalue of $A$
	with an eigenvector $0 \le x \in \dom{A}$.
	Thus, $Ax = \spb(A)x \ge 0$,
	so~\ref{cor:stability-for-pos-sg-quasi-compact:itm:small-gain} fails.
\end{proof}

Note that Stern's finite-dimensional results from \cite[Theorem~1.4]{Stern1982} can be seen as a special case 
of the equivalence of~\ref{cor:stability-for-pos-sg-quasi-compact:itm:stability} 
and~\ref{cor:stability-for-pos-sg-quasi-compact:itm:small-gain} 
in Corollary~\ref{cor:stability-for-pos-sg-quasi-compact} since the essential spectrum is always empty in finite dimensions.

\subsection{The spectral bound of operators with small essential spectrum}

While Corollary~\ref{cor:stability-for-pos-sg-quasi-compact} contains a characterization
of the property $\spb(A) < 0$, more generally one can instead also give a precise formula for the value of $\spb(A)$
as a consequence of the Krein--Rutman type Theorem~\ref{thm:krein-rutman-resolvent-positive}.
We do so in the following corollary.

\begin{corollary}
	Let $(X,X^+)$ be an ordered Banach space with total cone and 
	let $A: X \supseteq \dom{A} \to X$ be a resolvent positive linear operator;
	assume that $\specEss(A) = \emptyset$ or that $\spbEss(A) < \spb(A)$. 
	Then $\spb(A) < \infty$ and 
	\begin{align*}
		\spb(A) = \sup \{\lambda \in \R: \; \exists 0 \lneq x \in \dom{A} \text{ such that } Ax \ge \lambda x\}.
	\end{align*}
	If $\spb(A) > -\infty$, the supremum is even a maximum.
\end{corollary}
\begin{proof}
	In each of the cases $\specEss(A) = \emptyset$ and $\spbEss(A) < \spb(A)$ 
	it follows from Theorem~\ref{thm:krein-rutman-resolvent-positive} that $\spb(A) < \infty$.
	
	``$\ge$'' 
	Let $\lambda \in \R$ and let $0 \lneq x \in \dom{A}$ such that $Ax \ge \lambda x$, 
	i.e.\ $(\lambda - A)x \le 0$. 
	Assume for a contradiction that $\lambda > \spb(A)$.
	Then the resolvent $\Res(\lambda,A)$ is positive according to 
	Proposition~\ref{prop:properties-of-resolvent-positive-operators}%
	\ref{prop:properties-of-resolvent-positive-operators:itm:positive-and-decreasing} 
	and hence, $0 \le x = \Res(\lambda,A) (\lambda-A)x \le 0$, so $x = 0$;
	this contradicts $x \not= 0$.
	So we conclude that $\lambda \le \spb(A)$.
	
	``$\le$'' 
	If $\spb(A) = -\infty$, there is nothing to prove, so assume that $\spb(A) > -\infty$. 
	According to the assumptions we then have $\spb(A) > \spbEss(A)$. Thus, it follows from Theorem~\ref{thm:krein-rutman-resolvent-positive} that $\spb(A)$ 
	is an eigenvalue of $A$ with an eigenvector $0 \lneq u \in \dom{A}$. 
	Hence, $Au = \spb(A)u$, 
	so $\spb(A)$ is an element of the set under the supremum on the right,
	and we thus conclude that $\spb(A)$ is indeed the claimed maximum.
\end{proof}

Another formula for $\spb(A)$ under similar (though slightly stronger) 
assumptions than in the previous corolllary will be given in 
Theorem~\ref{thm:log-compact}.
For self-adjoint semigroups on Hilbert spaces, related results can also be found 
in Theorem~\ref{thm:log-ordered-hilbert} and Corollary~\ref{cor:s-a-hilbert}.

\section{A Collatz--Wielandt formula}
\label{section:collatz-wielandt}

For matrices $0 \le A \in \R^{d \times d}$ (where the inequality is meant entrywise) recall the so-called \emph{Collatz--Wielandt formula}
\begin{align}
	\label{eq:collatz-wielandt-spr-matrix}
	\spr(A) = \max_{0 \lneq x \in \R^d} \min_{j \in \{1, \dots, d\} \atop x_j \not= 0} \frac{(Ax)_j}{x_j}
\end{align}
for the spectral radius $\spr(A)$ of $A$, see for instance \cite[formula~(8.3.3) on p.\,670]{Meyer2000}. 
This can be easily generalized to yield a formula for the spectral bound of matrices whose off-diagonal entries are non-negative (i.e.\ matrices $A$ that satisfy $e^{tA} \ge 0$ for all $t \ge 0$):

\begin{proposition}
	\label{prop:collatz-wielandt-spb-matrix}
	Let $A = (A_{jk})_{j,k=1}^d \in \R^{d \times d}$ and assume that $A_{jk} \ge 0$ for all indices $j \not= k$. 
	Then
	\begin{align*}
		\spb(A) = \max_{0 \lneq x \in \R^d} \min_{j \in \{1, \dots, d\} \atop x_j \not= 0} \frac{(Ax)_j}{x_j}
	\end{align*}
\end{proposition}

The proposition can be derived by applying the classical Collatz--Wielandt formula~\eqref{eq:collatz-wielandt-spr-matrix} to the matrix $A+c\id$ for a sufficiently large real number $c$. 
However, we prefer to include a direct proof, since it will serve as a blueprint for an infinite-dimensional generalization of the formula in the sequel.

\begin{proof}[Proof of Proposition~\ref{prop:collatz-wielandt-spb-matrix}]
	Let $m(A)$ denote the maximum on the right-hand side.
	
	``$\le$'' 
	The number $\spb(A)$ is an eigenvalue of $A$ with an eigenvector $0 \lneq y \in \R^d$ 
	(this well-known fact is a finite-dimensional special case of Theorem~\ref{thm:krein-rutman-resolvent-positive}). 
	Hence,
	\begin{align*}
		 \spb(A) = \min_{j \in \{1, \dots, d\} \atop y_j \not= 0} \frac{(Ay)_j}{y_j} \le m(A). 
	\end{align*}
	
	``$\ge$'' 
	Let $0 \lneq x \in \R^d$ and assume for a contradiction that
	\begin{align*}
		\spb(A) < \min_{j \in \{1, \dots, d\} \atop x_j \not= 0} \frac{(Ax)_j}{x_j} =: \lambda.
	\end{align*}
	Note that we have $\lambda x \le Ax$. 
	Indeed, for those indices $j$ which satisfy $x_j \not= 0$, the inequality $\lambda x_j \le (Ax)_j$ follows from the definition of $\lambda$;
	and for those $j$ which satisfy $x_j = 0$, we have $\lambda x_j = 0 \le \sum_{j=1}^n A_{jk} x_k = (Ax)_j$, where the inequality in the middle is true since $A_{jk} \ge 0$ for $k\not=j$.
	
	As $e^{tA} \ge 0$ for all $t \ge 0$ it follows from $\lambda > \spb(A)$ that $\Res(\lambda,A) \ge 0$, 
	see Proposition~\ref{prop:properties-of-resolvent-positive-operators}%
	\ref{prop:properties-of-resolvent-positive-operators:itm:positive-and-decreasing}.
	We apply this to the inequality $(\lambda - A)x \le 0$ and thus obtain $x \le 0$; 
	this is a contradiction since $x \gneq 0$.
\end{proof}

Now we prove an infinite-dimensional version of this Collatz--Wielandt type formula.
It holds, for instance, for operators whose essential spectrum is empty 
(which is, for instance, satisfied for every operator with compact resolvent, 
see Proposition~\ref{prop:Essential-spectrum-simple-properties}(ii)).

Let $(X,X^+)$ be an ordered Banach space, let $A: X \supseteq \dom{A} \to X$ be a closed linear operator, 
and let $n \ge 0$ be an integer. 
We call a continuous functional $x': \dom{A^n} \to \R$ (where \emph{continuous} means continuous with respect to the graph norm on $\dom{A^n}$) \emph{positive} if $\langle x', x \rangle \ge 0$ for all $x \in \dom{A^n} \cap X^+$. 

\begin{definition}
	\label{def:set-determining-positivity} 
	Let $(X,X^+)$ be an ordered Banach space, let $A: X \supseteq \dom{A} \to X$ be a closed linear operator, 
	and let $n \ge 0$ be an integer. 
	We say that a set $F$ of continuous and positive functionals on $\dom{A^n}$ \emph{determines positivity} 
	if the following implication holds for each $x \in \dom{A^n}$:
	\begin{align*}
		\langle x', x \rangle \ge 0 \text{ for all } x' \in F
		\qquad 
		\Rightarrow 
		\qquad 
		x \ge 0.
	\end{align*}
\end{definition}

\begin{remark}
	\label{rem:Positiveaction-SDP} 
	If $F$ is as in the preceding definition and $0 \le x \in \dom{A^n}$ is non-zero, 
	then there exists at least one functional $x' \in F$ such that $\langle x', x \rangle > 0$; 
	for if $\langle x', x \rangle = 0$ for all $x' \in F$, then $x \ge 0$ as well as $-x \ge 0$ and hence $x = 0$.
\end{remark}

If $(X, X^+)$ is an ordered Banach space, then the set of all positive functionals in $X'$ always determines positivity 
(independently of the choice of $A$). 
Let us list a few examples of smaller sets that also determine positivity. 

\begin{examples}
	\begin{enumerate}[label=(\alph*)]
		\item 
		If $L$ is a locally compact Hausdorff space and $X = C_0(L)$, 
		then the set $F = \{\delta_\omega : \, \omega \in L\}$ of all point evaluations determines positivity 
		(no matter the choice of $A$). 
		
		\item 
		If $X = \ell^p$ for $p \in [1,\infty]$, then the set of coordinate functionals $\{e_k : \, k \in \N\}$, 
		where $e_k$ is defined as $\langle e_k, x \rangle := x_k$ for each $x \in X$, determines positivity 
		(again, for any operator $A$). 
		
		\item 
		Let $p \in (1,\infty)$, let $\Omega \subseteq \R^d$ be a bounded domain with, say, $C^\infty$-boundary
		and let $\Delta: L^p(\Omega) \supseteq \dom{\Delta} := W^{2,p}(\Omega) \cap W^{1,p}_0(\Omega) \to L^p(\Omega)$ 
		denote the Dirichlet Laplace operator on $L^p(\Omega)$. 
		For sufficiently large $n \in \N$ the domain $\dom{\Delta^n}$ embeds into the continuous functions $C_0(\Omega)$ 
		that vanish at the boundary of $\Omega$. 
		
		In this case, the set of point evaluations $\{\delta_x: \, x \in \Omega\}$ determines positivity. 
		This is an example where the choice of the operator matters since the point evaluations $\delta_x$ are not well-defined 
		on $L^p(\Omega)$, but only on $\dom{\Delta^n}$ for sufficiently large $n$.
		
		\item 
		Let $H$ be a complex Hilbert space and let $\calL(H)_{\operatorname{sa}}$ denote the (real) Banach space 
		of self-adjoint bounded linear operators on $H$, endowed with the cone of positive semi-definite operators. 
		
		For every $x \in H$, define the functional $\varphi_x: \calL(H)_{\operatorname{sa}} \to \R$ by 
		$\langle \varphi_x, C \rangle := \inner{x}{Cx}$ for all $C \in \calL(H)_{\operatorname{sa}}$. 
		Then the set $\{\varphi_x: \, x \in H\}$ determines positivity (independently of the choice of $A$).
	\end{enumerate}
\end{examples}

The following is our Collatz--Wielandt formula. 
It is reminiscent of the Donsker--Varadhan formula for the principal eigenvalue of second order differential operators 
\cite{DonskerVaradhan1976}; see also \cite{DonskerVaradhan1975}.

\begin{theorem}
	\label{thm:collatz-wieland-inf-dim-no-ess-spectrum}
	Let $A: X \supseteq \dom{A} \to X$ be the generator of a positive $C_0$-semigroup on an ordered Banach space $(X,X^+)$ with total cone $X^+$. 
	Assume that the essential spectrum $\specEss(A)$ is empty or that $\spbEss(A) < \spb(A)$.
	Let $n \ge 0$ be an integer and let $F \not= \emptyset$ be a set of positive and continuous functionals on $\dom{A^n}$ that determines positivity. 
	Then $\spb(A) < \infty$ and the formula
	\begin{align*}
		\spb(A) = \sup_{0 \lneq x \in \dom{A^{n+1}}} \inf_{x' \in F \atop \langle x', x \rangle \not= 0} \frac{\langle x', Ax \rangle }{\langle x', x \rangle}
	\end{align*}
	holds. 
\end{theorem}

Note that, in case that $\spbEss(A) = \emptyset$, it may happen that both sides of the equality are equal to $-\infty$.
We remark that Collatz--Wielandt type results for positive linear operators (rather than for generators of positive semigroups) are studied in \cite{Friedland1990, Marek1992, Friedland2020}.

\begin{proof}[Proof of Theorem~\ref{thm:collatz-wieland-inf-dim-no-ess-spectrum}]
	First note that indeed $\spb(A) < \infty$: 
	if $\spbEss(A) < \spb(A)$ this follows directly from Theorem~\ref{thm:krein-rutman-resolvent-positive};
	and if $\specEss(A) = \emptyset$, then $\spbEss(A) = -\infty$, so it also follows from
	Theorem~\ref{thm:krein-rutman-resolvent-positive} that $\spb(A) < \infty$.

	Now let $m(A)$ denote the supremum on the right-hand side. 
	We essentially follow the proof of Proposition~\ref{prop:collatz-wielandt-spb-matrix}.
	
	``$\le$'' 
	If $\spb(A) = -\infty$, this inequality is trivial, so assume that $\spb(A) > -\infty$. 
	Then we have $\spbEss(A) < \spb(A)$ in any case. 
	So according to the Krein--Rutman type Theorem~\ref{thm:krein-rutman-resolvent-positive}, $\spb(A)$ is an eigenvalue of $A$ with an eigenvector $y \in X^+$. 
	Since every eigenvector of $A$ is in $\dom{A^{n+1}}$, it follows that
	\begin{align*}
		\spb(A) = \inf_{x' \in F \atop \langle x', y \rangle \not= 0} \frac{\langle x', Ay \rangle }{\langle x', y \rangle} \le m(A).
	\end{align*}
	Here we used that, as mentioned in Remark~\ref{rem:Positiveaction-SDP}, there exists at least one functional $x' \in F$ such that $\langle x', y \rangle \not= 0$.
	
	``$\ge$'' 
	Assume for a contradiction that there exists $0 \lneq x \in \dom{A^{n+1}}$ such that
	\begin{align*}
		\spb(A) < \inf_{x' \in F \atop \langle x', x \rangle \not= 0} \frac{\langle x', Ax \rangle }{\langle x', x \rangle} =: \lambda.
	\end{align*}
	Then we have $\lambda x \le A x$. 
	Indeed, for those $x' \in F$ that satisfy $\langle x',x \rangle \not= 0$ the inequality $\langle x', \lambda x \rangle \le \langle x', A x \rangle$ follows from the definition of $\lambda$. 
	And for those $x' \in F'$ that satisfy $\langle x', x\rangle=0$ it follows from the positivity of the semigroup $(e^{tA})_{t \ge 0}$ that 
	\begin{align*}
		\langle x', Ax \rangle = \lim_{t \downarrow 0} \frac{\langle x', e^{tA}x \rangle }{t} \ge 0 = \langle x', \lambda x \rangle.
	\end{align*}
	For the first equality we used that $Ax = \lim_{t \downarrow 0} \frac{e^{tA}x - x}{t}$ with respect to the graph norm in $\dom{A^n}$ since $x \in \dom{A^{n+1}}$.
	As $F$ determines positivity, it thus follows that $\lambda x \le Ax$, as claimed.

	By assumption, $\spb(A) < \lambda$.
	Hence, Proposition~\ref{prop:properties-of-resolvent-positive-operators}\ref{prop:properties-of-resolvent-positive-operators:itm:positive-and-decreasing} 
	ensures that the resolvent $\Res(\lambda,A)$ is positive.
	We now apply this operator to the inequality $(\lambda - A)x \le 0$ and thus obtain $x \le 0$, 
	which is a contradiction to $x \gneq 0$.
\end{proof}

Let us describe three more concrete situations where this result can be applied.
The first one is a direct generalization of Proposition~\ref{prop:collatz-wielandt-spb-matrix} to sequence spaces.

\begin{corollary}
	Let $A$ be the generator of a positive $C_0$-semigroup on $\ell^p$ for $1 \le p < \infty$, 
	and assume that $\specEss(A) = \emptyset$ or that $\spbEss(A) < \spb(A)$. 
	Then $\spb(A) < \infty$ and
	\begin{align*}
		\spb(A) = \sup_{0 \lneq x \in \dom{A}} \inf_{j \in \N \atop x_j \not= 0} \frac{(Ax)_j}{x_j}.
	\end{align*}
\end{corollary}

\begin{proof}
	This follows directly from Theorem~\ref{thm:collatz-wieland-inf-dim-no-ess-spectrum} 
	if we chose $n=0$ and $F \subseteq (\ell^p)'$ to be the set of all canonical unit vectors.
\end{proof}

Note that, in the situation of the corollary, the inequality $\spb(A) < \infty$ follows 
not only from Theorem~\ref{thm:collatz-wieland-inf-dim-no-ess-spectrum} 
but also more directly from 
Proposition~\ref{prop:properties-of-resolvent-positive-operators}%
\ref{prop:properties-of-resolvent-positive-operators:itm:spectral-bound} 
(as the cone in $\ell^p$ is generating and normal).

The second concrete situation that we discuss are semigroups on spaces of continuous functions.

\begin{corollary}
	Let $A$ be the genereator of a positive $C_0$-semigroup on $C_0(L)$, where $L$ is a locally compact Hausdorff space.
	Assume that $\specEss(A) = \emptyset$ or that $\spbEss(A) < \spb(A)$. 
	Then $\spb(A) < \infty$ and 
	\begin{align*}
		\spb(A) = \sup_{0 \lneq u \in \dom{A}} \inf_{\omega \in L \atop u(\omega) \not= 0} \frac{(Au)(\omega)}{u(\omega)}.
	\end{align*}
\end{corollary}

\begin{proof}
	This follows directly from Theorem~\ref{thm:collatz-wieland-inf-dim-no-ess-spectrum} 
	if we chose $n=0$ and $F \subseteq (C_0(L))'$ to be the set of all point evaluation maps.
\end{proof}

As in the previous corollary, the inequality $\spb(A) < \infty$ can also be directly inferred from 
Proposition~\ref{prop:properties-of-resolvent-positive-operators}%
\ref{prop:properties-of-resolvent-positive-operators:itm:spectral-bound} 
(alternatively to Theorem~\ref{thm:collatz-wieland-inf-dim-no-ess-spectrum})
here, as the cone in $C_0(L)$ is normal and generating.

The third situation that we consider is a more specific example: 
We perturb the Neumann Laplace operator on a bounded interval by a potential 
and give criteria for its spectral bound to be strictly positive, 
even if the potential has a large negative part.

\begin{example}
	\label{exa:neumann-laplace-with-potential}
	Let $\emptyset \not= (a,b) \subseteq \R$ be a bounded and open interval 
	and let $\Delta: L^2(a,b) \supseteq \dom{\Delta} \to L^2(a,b)$ denote the Neumann Laplace operator with domain
	\begin{align*}
		\dom{\Delta} := \{u \in H^2(a,b): \; u'(a) = u'(b) = 0\}
	\end{align*}
	given by $\Delta u = u''$.
	Let $V \in C([a,b])$ and consider the operator $A := \Delta + V$ with domain $\dom{A} := \dom{\Delta}$.
	Then
	\begin{align}
		\label{eq:exa:neumann-laplace-with-potential:s-a-formula}
		\spb(A) = \sup_{0 \lneq u \in \dom{A^2}} \inf_{\omega \in [a,b] \atop u(\omega) \not= 0} \Big(\frac{u''(\omega)}{u(\omega)} + V(\omega)\Big).
	\end{align}
	To see this, first notice that $A$ generates a positive $C_0$-semigroup. 
	Indeed, it is well-known that the Neumann Laplace operator $\Delta$ generates a positive semigroup on $L^2(a,b)$.
	Since $V$ is a bounded perturbation, $A$ also generates a $C_0$-semigroup $(e^{tA})_{t \ge 0}$
	(\cite[Theorem~III.1.3 on p.\,158]{EngelNagel2000})
	and this semigroup is positive since we have
	\begin{align*}
		e^{tA} = e^{-\norm{V}_\infty} e^{t(\Delta+V+\norm{V}_\infty)} \ge 0
	\end{align*}
	for all $t \ge 0$, where the inequality at the end follows from the Dyson--Phillips series expansion
	since the operator $V+\norm{V}_\infty$ is positive
	(see e.g.\ \cite[Proposition~11.6 and Corollary~11.7]{BatkaiKramarRhandi2017} for details).
	
	Since the semigroup generated by $A$ is positive and since $A$ has empty essential spectrum
	(indeed, $A$ has compact resolvents as $\dom{A} \subseteq H^2(a,b)$ embeds compactly into $L^2(a,b)$),
	we can apply Theorem~\ref{thm:collatz-wieland-inf-dim-no-ess-spectrum}:
	in this theorem, choose $n=1$ and let $F$ be the set of all point evaluations at points in $[0,1]$;
	then we immediately obtain the claimed formula~\eqref{eq:exa:neumann-laplace-with-potential:s-a-formula}.
	(Note that it is useful here to have the case $n=1$ available in the theorem: 
	point evaluations are not well-defined on $L^2(0,1)$, 
	but they are well-defined on $\dom{A} = \dom{\Delta} \subseteq H^2(0,1)$.)
	
	In the rest of this example, we will study the following question in two interesting cases, 
	and formula~\eqref{eq:exa:neumann-laplace-with-potential:s-a-formula} will turn out to be useful in the second case:
	
	\begin{center}
		Under which conditions on the potential $V$ \\
		does the operator $A$ have strictly positive spectral bound?
	\end{center}
	
	\textbf{1st case:} $V$ has integral $0$ but $V \not= 0$.
	
	We will show that always $\spb(A) > 0$ in this case.
	This might be slightly surprising at first glance since the Neumann Laplace operator 
	does have spectral bound $0$ -- so due to the assumption that $V$ has integral $0$,
	it does not seem clear at first glance how $V$ influences the spectral bound.
		
	To prove that $\spb(A) > 0$, 
	first note that, as the constant function $\one$ is in the kernel of $\Delta$ and in the domain of $A$, 
	we have $\langle A\one, \one \rangle = \int_a^b V(\omega) \dx \omega = 0$.
	So $0$ is in the numerical range of $A$, 
	and since $A$ is self-adjoint it follows, 
	 for instance from the spectral theorem for self-adjoint operators,
	that $0$ is in the convex hull of $\spec(A)$.
	So $\spb(A) \ge 0$.
	
	Now assume for a contradiction that $\spb(A) = 0$. 
	Then, for instance by considering the spectral expansion of $A$ 
	(for self-adjoint operators with compact resolvent) 
	we see that the equality $\langle A\one, \one \rangle = 0$ implies that $\one$ is an eigenvector of $A$.
	Hence, $0 = A\one = V\one = V$, which is a contradiction since we assumed that $V \not= 0$.
	
	\textbf{2nd case:} 
	$V$ is bounded above and its integral is close to $-\infty$. 
	
	For this case we are going to show the following existence result:
	For the interval $(a,b) = (-\pi/2, \pi/2)$ 
	we can find a $C^\infty$-function $V$ on $[-\pi/2,\pi/2]$
	whose integral is arbitrarily close to $-\infty$ and which satisfies $V \le 3$ and $s(A) \ge 1$.
	
	Again, this is somewhat surprising since one might expect a very negative integral value of $V$
	to draw the spectral bound of $A$ to the left rather than to the right.
	
	A specific class of examples where this occurs can be obtained as follows.
	For the potential $V$, which is yet to be determined, consider the function
	\begin{align*}
		h: \omega \mapsto \frac{2}{\sin^2 \omega} - 4 + V(\omega)
	\end{align*}
	for $\omega \in [-\pi/2, \pi/2] \setminus \{0\}$.
	Close to $0$ the function $\omega \mapsto \frac{2}{\sin^2 \omega}$
	has a non-integrable singularity at $0$ 
	(as it behaves asymptotically like $\frac{2}{\omega^2}$ there).
	
	Hence, we can choose a $C^\infty$-function $V$ on $[-\pi/2, \pi/2]$ 
	which has the following properties:
	(a) $V \le 3$ everywhere;
	(b) $h(\omega) \ge 1$ for all $\omega \not= 0$;
	(c) $V$ has a sharp negative peak close to $0$ and its integral is arbitrarily close to $-\infty$.
	
	In order to estimate $\spb(A)$ for our choice of $V$, 
	we use formula~\eqref{eq:exa:neumann-laplace-with-potential:s-a-formula} with the function
	$u := \sin^2$. 
	By a brief computation one can check that $u \in \dom{\Delta}$ as well as $\Delta u \in \dom{\Delta}$ 
	and $Vu \in \dom{\Delta}$ 
	(since $V$ is smooth and since both $u$ and its derivatives vanish at the boundary of the interval). 
	Hence, formula~\eqref{eq:exa:neumann-laplace-with-potential:s-a-formula} yields
	\begin{align*}
		\spb(A) 
		\ge 
		\inf_{\omega \in [-\pi/2, \pi/2] \atop \omega \not= 0} 
		\Big(
			\underbrace{
				\frac{u''(\omega)}{u(\omega)} + V(\omega)
			}_{= h(\omega)}
		\Big)
		\ge 
		1,
	\end{align*}
	as claimed.
\end{example}

Of course, numerous results about the spectral bound of the operator $\Delta + V$
(or equivalently, about the ground energy of the Schrödinger operator $-\Delta + \tilde V$, with $\tilde V := -V$, 
to use the sign convention that is common in mathematical physics) 
are available in the literature.
For instance, with regard to the first case in the example above, we also point to a result by Simon 
on the whole space $\R$ rather than on bounded intervals \cite[Theorem~2.5]{Simon1976}.

\begin{remark}
	Theorem~\ref{thm:collatz-wieland-inf-dim-no-ess-spectrum} can be generalized to operators which do not generate semigroups, 
	but are densely defined and resolvent positive and satisfy the Hille--Yosida type estimate 
	\begin{align*}
		\limsup_{\lambda \to \infty} \norm{\lambda \Res(\lambda,A)} < \infty
	\end{align*}
	Instead of the formula $Ax = \lim_{t \downarrow 0} \frac{e^{tA} x - x}{t}$ in $\dom{A^n}$ for $x \in \dom{A^{n+1}}$
	one then has to use the subsequent Lemma~\ref{lem:approximation-of-A} in the proof.
	(Note that the lemma can be transferred to vectors $x \in \dom{A^{n+1}}$ and convergence in $\dom{A^n}$ 
	by multiplying with $(\mu - A)^n$ from the right and with $\Res(\mu,A)^n$ from the left, 
	for an arbitrary point $\mu$ in the resolvent set of $A$.)
\end{remark}

\begin{lemma}
	\label{lem:approximation-of-A}
	Let $A: X \supseteq \dom{A} \to X$ be a densely defined linear operator on a Banach space $X$, 
	and assume that there exist real numbers $M \ge 1$ and $\lambda_0$ 
	such that every real number $\lambda > \lambda_0$ is in the resolvent set of $A$ 
	and satisfies $\norm{\lambda \Res(\lambda,A)} \le M$.
	
	Then for every $x \in \dom{A}$ we have $\lambda^2 \Res(\lambda,A) x - \lambda x \to Ax$ 
	with respect to the norm in $X$ as $\lambda \to \infty$.
\end{lemma}
\begin{proof} 
	We first note that $A\Res(\lambda,A)$ converges strongly to $0$ on $X$ as $\lambda \to \infty$. 
	Indeed, for $y \in \dom{A}$ one has
	\begin{align*}
		\norm{A \Res(\lambda,A) y} 
		= 
		\norm{\Res(\lambda,A) Ay} 
		\le 
		\frac{M}{\lambda} \norm{Ay}
		\to 0
	\end{align*}
	as $\lambda \to \infty$. 
	Moreover, the operators $A\Res(\lambda,A) \in \calL(X)$ are uniformly bounded as $\lambda \to \infty$, 
	since $A\Res(\lambda,A) = -\id + \lambda \Res(\lambda,A)$ has norm at most $1+M$ for large $\lambda$.
	Hence the claimed strong convergence follows from the density of $\dom{A}$ in $X$.
	
	Moreover, the equality $\lambda\Res(\lambda,A) - \id = A \Res(\lambda,A)$ implies that, for $x \in \dom{A}$,  
	\begin{align*}
		\lambda^2 \Res(\lambda,A) x - \lambda x 
		= 
		\lambda \Res(\lambda,A) Ax
		= 
		Ax + A \Res(\lambda, A) Ax
		\to Ax
	\end{align*}
	as $\lambda \to \infty$, as claimed.
\end{proof}

\section{Logarithmic formulas for the spectral bound}
\label{section:logarithmic}

In this final section, we show how the spectral bound of the generator of a positive semigroup
can, under appropriate assumptions, be computed from a single orbit of the semigroup.
For self-adjoint positive semigroups on $L^2$-spaces 
this was shown in \cite[Theorem~2.1]{KellerLenzVogtWojciechowski2015};
we will generalize this result in Theorem~\ref{thm:log-ordered-hilbert} 
to more general classes of ordered Hilbert spaces, 
and we will give another version of the result in Theorem~\ref{thm:log-compact}
which does not need any Hilbert space structure at all, but requires stronger assumptions 
on the semigroup instead.

In order to state and prove our results, we need to discuss several notions 
that are related to the idea of ``strict (or strong) positivity'' of a vector.
There are several ways to make this concept precise, and we recall them in the following definition.

\begin{definition}
	Let $(X, X^+)$ be an ordered Banach space.
	\begin{enumerate}[label=(\alph*)]
		\item 
		A vector $u \in X^+$ is called a \emph{quasi-interior point} of $X^+$ 
		if the vector subspace $\bigcup_{n \in \N} [-nu, nu]$ of $X$ is dense in $X$.
		
		\item 
		A vector $u \in X^+$ is called an \emph{almost interior point} of $X^+$ 
		if $\langle x', u \rangle > 0$ for every non-zero functional $x' \in (X')^+$.
		
		\item 
		A functional $u' \in (X^+)'$ is called \emph{strictly positive} 
		if $\langle u', x \rangle$ is non-zero for every non-zero vector $x \in X^+$.
	\end{enumerate}
\end{definition}

The notions of quasi-interior and almost interior points are subtle. 
It is easy to see that every quasi-interior point is automatically almost interior, 
and the converse implication is known to be false, in general, see
\cite[p.\,136]{Schaefer1960} and \cite[Section~3.6]{KrasnoselskiiLifshitsSobolev1989}. 
However, in all known counterexamples the positive cone is only total, but not generating.
If the cone is generating (or even more, generating and normal), 
it is, to the best of our knowledge, open whether the concepts of quasi-interior points 
and almost interior points coincide 
(see \cite[Open Problem~2.5]{GlueckWeber2020}). 
However, they are known to coincide in each of the following cases: 
(a) if $(X, X^+)$ is a Banach lattice
\cite[Theorem~II.6.3(a) and~(c)]{Schaefer1974};
(b) if, more generally, $(X, X^+)$ has normal and generating cone and the Riesz decomposition property 
\cite[Theorem~6]{KatsikisPolyrakis2006};
(c) if $(X, X^+)$ is the self-adjoint part of a $C^*$-algebra, 
endowed with the cone of positive semi-definite elements
\cite[Example~2.15(i)]{GlueckWeber2020}; 
(d) if the positive cone $X^+$ has non-empty interior 
-- in this case, a point $u \in X^+$ is an interior point 
if and only if it is a quasi-interior point 
if and only if it is an almost-interior point 
\cite[Corollary~2.8]{GlueckWeber2020}.

The reason why quasi-interior and almost-interior points are useful for our purposes 
are the characterizations in the following two propositions.
Recall that \emph{order intervals} were defined in Subsection~\ref{subsection:ordered-banach-spaces}.

\begin{proposition}
	\label{prop:char-quasi-interior}
	Let $(X,X^+)$ be an ordered Banach space. 
	For $u \in X^+$ the following are equivalent:
	\begin{enumerate}[label=\upshape(\roman*)]
		\item\label{prop:char-quasi-interior:itm:quasi-interior} 
		The vector $u$ is a quasi-interior point of $X^+$.
		
		\item\label{prop:char-quasi-interior:itm:functional-interval} 
		If a functional $x' \in X'$ vanishes on the order interval $[0,u]$, 
		then $x' = 0$. 
		
		\item\label{prop:char-quasi-interior:itm:operator-interval} 
		If $Y$ is a Banach space and a bounded linear operator $T: X \to Y$ 
		vanishes on the order interval $[0,u]$, then $T = 0$.
	\end{enumerate}
\end{proposition}

\begin{proof}
	All equivalences immediately follow from the fact that the span of the order interval $[0,u]$
	is equal to the linear space $U=\bigcup_{n \in \N} [-nu, nu]$, 
	and from the well-known corollary of the Hahn--Banach theorem that $U$ is dense in $X$ if and only if  
	for every functional $x' \in X'$ the condition $\langle x', x \rangle = 0$ for all $x \in U$ implies $x' = 0$.
\end{proof}

While the above proposition is almost trivial, it still sheds some light to the distinction between 
quasi-interior and almost interior points when compared to the following 
characterization of the latter concept; 
the main difference is that, now, one has to restrict the attention to
positive operators and functionals in the equivalent conditions.

\begin{proposition}
	\label{prop:char-almost-interior}
	Let $(X,X^+)$ be an ordered Banach space. 
	For $u \in X^+$ the following are equivalent:
	\begin{enumerate}[label=\upshape(\roman*)]
		\item\label{prop:char-almost-interior:itm:almost-interior} 
		The vector $u$ is an almost interior point of $X^+$.
		
		\item\label{prop:char-almost-interior:itm:functional-interval} 
		If a positive functional $x' \in (X')^+$ vanishes on the order interval $[0,u]$, 
		then $x' = 0$. 
		
		\item\label{prop:char-almost-interior:itm:operator-interval} 
		If $(Y,Y^+)$ is an ordered Banach space and a positive bounded linear operator $T: X \to Y$ 
		vanishes on the order interval $[0,u]$, then $T = 0$.
		
		\item\label{prop:char-almost-interior:itm:operator-point} 
		If $(Y,Y^+)$ is an ordered Banach space and a positive bounded linear operator $T: X \to Y$ 
		satisfies $Tu = 0$, then $T = 0$.
	\end{enumerate}
\end{proposition}

\begin{proof}
	\Implies{prop:char-almost-interior:itm:almost-interior}{prop:char-almost-interior:itm:operator-point}
	Let $(Y,Y^+)$ and $T$ be as in~\ref{prop:char-almost-interior:itm:operator-point}
	and assume that $Tu = 0$. 
	For every positive $y' \in (Y')^+$ we then have 
	$\langle T' y', u \rangle = \langle y', Tu \rangle = 0$ and thus, 
	as $T'y'$ is a positive functional and $u$ is an almost interior point of $X^+$, 
	we have have $T'y' = 0$. 
	But since $Y^+$ is a cone in $Y$, the span of the positive functionals on $Y$ 
	is weak${}^*$-dense in $Y'$. 
	Thus, $T' y' = 0$ for each $y' \in Y'$, so we conclude that $T' = 0$ and thus $T = 0$.
	
	\Implies{prop:char-almost-interior:itm:operator-point}{prop:char-almost-interior:itm:operator-interval}
	This implication is obvious.
	
	\Implies{prop:char-almost-interior:itm:operator-interval}{prop:char-almost-interior:itm:functional-interval}
	This implication is obvious.
	
	\Implies{prop:char-almost-interior:itm:functional-interval}{prop:char-almost-interior:itm:almost-interior}
	Let $x' \in (X')^+$ be a positive functional such that $\langle x', u \rangle = 0$.
	For every $x \in [0,u]$ it then follows from the positivity of $x'$ that
	\begin{align*}
		0 \le \langle x', x \rangle \le \langle x', u \rangle = 0,
	\end{align*}
	so $\langle x', x \rangle = 0$. 
	According to~\ref{prop:char-almost-interior:itm:functional-interval} 
	this implies that $x' = 0$, 
	so $u$ is indeed an almost interior point of $X^+$.
\end{proof}

Now we can prove the first main result of this section. 
Recall that a $C_0$-semigroup $(e^{tA})_{t \ge 0}$ on a Banach space is called 
\emph{eventually compact} if 
the operator $e^{tA}$ is compact for some $t \ge 0$ (and hence, for all subsequent $t$).
	
\begin{theorem}
	\label{thm:log-compact}
	Let $(e^{tA})_{t \ge 0}$ be a positive and eventually compact $C_0$-semigroup 
	with generator $A$ on an ordered Banach space $(X,X^+)$. 
	If $u \in X^+$ is an almost interior point and $u' \in (X')^+$ is a strictly positive functional,
	then
	\begin{align}
		\label{eq:Formula-for-s(A)}
		\spb(A) = \omega(A) = \lim_{t \to \infty} \frac{\log \langle u', e^{tA} u \rangle}{t}.
	\end{align}
\end{theorem}

\begin{proof}
	The equality $\spb(A) = \omega(A)$ is true for each eventually compact $C_0$-semigroup 
	on every Banach space, see \cite[Corollary~IV.3.12 on p.\,281]{EngelNagel2000}, 
	so we only need to show the equality on the right.

	Note that for every $\lambda>\omega(A)$ there exist numbers $M,C>0$ such that
	\[
		\limsup_{t \to \infty} \frac{\log \langle u', e^{tA} u \rangle}{t}
		\leq \limsup_{t \to \infty} \frac{\log (\|u'\| Me^{\lambda t} \|u\|)}{t}
		\leq \limsup_{t \to \infty} \frac{C+\lambda t}{t} = \lambda.
	\]
	Thus, we have (for not necessarily eventually compact semigroups) that
	\begin{align}
		\label{eq:Formula-for-s(A)-inequality}
		\omega(A) \geq \limsup_{t \to \infty} \frac{\log \langle u', e^{tA} u \rangle}{t},
	\end{align}
	This implies that if $\omega(A)=-\infty$, 
	then the right hand side in \eqref{eq:Formula-for-s(A)} is equal to $-\infty$ as well. 

	Assume now that $\omega(A) \in \R$. 	
	Let $(t_n) \subseteq [0,\infty)$ be a sequence of times that converges to $\infty$. 
	In view of \eqref{eq:Formula-for-s(A)-inequality}, to prove \eqref{eq:Formula-for-s(A)}
	it suffices to show the existence of a subsequence $(t_{n_r})$ such that 
	\begin{align*}
		\frac{\log \langle u', e^{t_{n_r}A} u \rangle}{t_{n_r}} \to \omega(A).
	\end{align*}
	To this end, we use the eventual compactness of the semigroup: 
	it implies, see \cite[Corollary~V.3.2 on pp.\,330--331]{EngelNagel2000},
	that there exist finitely many numbers $i\beta_1, \dots, i\beta_\ell \in i\R$,
	non-zero bounded linear operators $Q_1, \dots, Q_\ell$,
	and an integer $k \ge 0$, such that
	\begin{align*}
		t^{-k} e^{-t\omega(A)}e^{tA} 
		= 
		\sum_{j=1}^\ell e^{i\beta_j t} \, Q_j  +  S(t)
	\end{align*}
	for all $t > 0$, 
	where the $S(t)$ are bounded linear operators for $t \ge 0$ which satisfy $\norm{S(t)} \to 0$ 
	as $t \to \infty$; 
	moreover, for each $j \in \{1, \dots, \ell\}$ the number $\omega(A) + i\beta_j$ is an isolated spectral value 
	of $A$ and for the associated spectral projections $P_1, \dots, P_\ell$
	we have $P_j Q_h = 0$ for $j \not= h$ and $P_j Q_j = Q_j$ for each $j$.
	
	We can find a subsequence $(t_{n_r})$ of $(t_n)$ such that, for each $j \in \{1, \dots, \ell\}$, 
	the sequence $(e^{i\beta_j t_{n_r}})$ converges to a point $\mu_j$ on the complex unit circle. 
	Thus, 
	\begin{align*}
		t_{n_r}^{-k} e^{-t_{n_r}\omega(A)}e^{t_{n_r}A} 
		\to 
		\sum_{j=1}^\ell \mu_j Q_j 
		=: 
		Q
	\end{align*}
	Clearly, the operator $Q$ is positive as a limit of positive operators.
	Moreover, $Q$ is non-zero since, for instance, $P_1 Q = \mu_1 Q_1 \not= 0$. 
	As $u$ is an almost interior point of $X^+$ it thus follows from 
	Proposition~\ref{prop:char-almost-interior}\ref{prop:char-almost-interior:itm:operator-point} 
	that $Qu \not= 0$;
	and as $u'$ is a strictly positive functional, we thus conclude that $\langle u', Qu \rangle \not= 0$.
	Hence, the limit $\log(\langle u', Qu \rangle)$ of
	\begin{align*}
		\log
		\Big(
			t_{n_r}^{-k} e^{-t_{n_r}\omega(A)} \langle u', e^{t_{n_r}A} u \rangle
		\Big)
		=
		-k \log t_{n_r} - t_{n_r} \omega(A) + \log\langle u', e^{t_{n_r}A} u \rangle 
	\end{align*}
	is a real number (rather than $-\infty$),
	and thus
	\begin{align*}
		\frac{\log\langle u', e^{t_{n_r}A} u \rangle}{t_{n_r}}
	\end{align*}
	converges to $\omega(A)$, as desired.
\end{proof}

Now we come to the second main result of this section. 
By an \emph{ordered Hilbert space} we mean an ordered Banach space $(H, H^+)$ endowed with an inner product that induces 
the norm on $H$. If we identify $H$ with its dual space by means of the Riesz--Fréchet representation theorem, 
the dual cone $(H')^+$ becomes a subset of $H$ and is thus given by
\begin{align*}
	(H')^+ = \{x \in H: \; \inner{x}{y} \ge 0 \text{ for all } y \in H^+\}.
\end{align*}
We use this identification in the formulation of the following theorem.
Note that, if $H = L^2(\Omega,\mu)$ for some measure space $(\Omega,\mu)$ 
is endowed with the pointwise almost everywhere order, 
then $H^+ = (H')^+$, i.e., the cone is \emph{self-dual}. 
Self-dual cones on Hilbert spaces have been studied on various occasions in the literature; 
see for instance the classical paper \cite{Penney1976}.

The following theorem is a generalization of \cite[Theorem~2.1]{KellerLenzVogtWojciechowski2015}, 
where the result was prove for $L^2$-spaces with their usual cone
(see the discussion after the theorem for more details).

\begin{theorem}
	\label{thm:log-ordered-hilbert}
	Let $(e^{tA})_{t \ge 0}$ be a positive and self-adjoint $C_0$-semigroup 
	with generator $A$ on a non-zero ordered Hilbert space $(H,H^+)$, 
	and assume that $H^+ \subseteq (H')^+$.
	Let $u,v \in H^+$ and assume that there exists a quasi-intererior point $w \in H^+$ such that
	$u,v \ge w$.
	Then 
	\begin{align*}
		-\infty < \spb(A) = \omega(A) = \lim_{t \to \infty} \frac{\log \inner{v \,}{\, e^{tA} u}}{t}.
	\end{align*}
\end{theorem}

A few comments are in order before we prove the theorem. 
The assumptions that $u,v \ge w$ and that $w$ be a quasi-interior point immediately implies that $u,v$ 
are quasi-interior points, too. 
We do not know if, conversely, for all quasi-interior points $u,v \in H^+$ there always exists a quasi-interior point 
$w \in H^+$ that satisfies $u,v \ge w$.
However, if $H = L^2(\Omega,\mu)$ for a $\sigma$-finite measure space $(\Omega,\mu)$ is endowed its usual order,
then the quasi interior points of $H^+$ are precisely those functions which are strictly positive almost everywhere.
Hence, for two quasi-interior points $u,v$ it then follows that the infimum $w := u \land v$ is also a quasi-interior point, 
and this function clearly satisfies $u,v\ge w$.
The latter observation is the reason why, on $L^2$-spaces, it suffices to assume that $u$ and $v$ are quasi-interior points. 
Hence, the $L^2$-version of the theorem 
in \cite[Theorem~2.1]{KellerLenzVogtWojciechowski2015} has slightly simpler assumptions.

For the proof of the theorem, we need one more ingredient --
namely, a version of Proposition~\ref{prop:char-quasi-interior} for symmetric bilinear mappings:

\begin{proposition}
	\label{prop:char-quasi-interior-bilinear}
	Let $(X,X^+)$ be an ordered Banach space, let $Y$ be a Banach space 
	and let $b: X \times X \to Y$ be a continuous bilinear mapping which is symmetric
	(in the sense that $b(x_1,x_2) = b(x_2,x_1)$ for all $x_1,x_2 \in X$). 
	
	If $b$ is non-zero and $u \in X^+$ is a quasi-interior point, 
	then there exists a point $\tilde u \in [0,u]$ such that $b(\tilde u, \tilde u) \not= 0$.
\end{proposition}

\begin{proof}
	Assume that $b(\tilde u, \tilde u) = 0$ for all $\tilde u \in [0,u]$. 
	It immediately follows that $b(\tilde u, \tilde u) = 0$ for all $\tilde u \in C := \bigcup_{n \in \N} [0, nu]$.
	Next we show that the same is true even for all $\tilde u \in \bigcup_{n \in \N} [-nu,nu]$.
	So fix such a $\tilde u$; we can write it as $\tilde u = \tilde v - \tilde w$ 
	for some $\tilde v, \tilde w \in C$.
	The real polynomial function
	\begin{align*}
		t \mapsto b(\tilde v - t\tilde w, \tilde v - t \tilde w)
	\end{align*}
	vanishes for $t \in (-\infty,0]$ since $\tilde v - t \tilde w \in C$ for all these $t$. 
	By the identity theorem for polynomials, we conclude that 
	the polynomial is identically $0$, so in particular for $t = 1$ we obtain, $b(\tilde u, \tilde u) = 0$.
	
	As $u$ is a quasi-interior point of $X_+$ the set $\bigcup_{n \in \N} [-nu,nu]$ is dense in $X$. 
	Due to the continuity of $b$ it thus follows that $b(x,x) = 0$ for all $x \in X$. 
	Finally, as $b$ is symmetric, it follows from the polarization identity
	\[
		2b(x_1,x_2) = b(x_1+x_2,x_1+x_2)-b(x_2,x_2)-b(x_1,x_1),
	\]	
	that $b(x_1,x_2) = 0$ for all $x_1,x_2 \in X$, which contradicts the assumption $b \not= 0$.
\end{proof}

\begin{proof}[Proof of Theorem~\ref{thm:log-ordered-hilbert}]
	Self-adjoint operators on non-zero spaces always have non-empty spectrum, so $\spb(A) >- \infty$.
	Moreover, the equality $\spb(A) = \omega(A)$ holds for all analytic semigroups and thus, 
	in particular, for all self-adjoint semigroups.
	To prove the remaining equality we may, and shall, assume that $\spb(A) = \omega(A) = 0$.
	We need to show that
	\begin{align*}
		\limsup_{t \to \infty} \frac{\log \inner{v \,}{\, e^{tA} u}}{t} 
		\le 
		0 
		\le 
		\liminf_{t \to \infty} \frac{\log \inner{v \,}{\, e^{tA} u}}{t}
	\end{align*}
	The first inequality readily follows from $\omega(A) = 0$.
	To show the second inequality, 
	we follow the main idea of the proof given in \cite[Theorem~2.1]{KellerLenzVogtWojciechowski2015}, 
	although we present the details in a somewhat different manner:
	by the multiplier version of the spectral theorem for self-adjoint operators 
	there exists a measurable space $(\Omega,\mu)$, a measurable function $m: \Omega \to (-\infty,0]$,
	and a unitary operator $U: H \to L^2(\Omega,\mu)$ such that 
	$A = U^* M_m U$, where $M_m: L^2(\Omega,\mu) \supseteq \dom{M_m} \to L^2(\Omega,\mu)$ 
	is the operator that acts as the multiplication with $m$.%
	\footnote{I.e., $\dom{M_m} = \{f \in L^2(\Omega,\mu) : \; mf \in L^2(\Omega,\mu)\}$ 
	and $M_m f = mf$ for all $f \in \dom{M_m}$.}%
	\footnote{We note that the semigroup generated by $M_m$ on $L^2(\Omega,\mu)$, i.e., 
	$(M_{e^{tm}})_{t \ge 0}$, is conjugate to $(e^{tA})_{t \ge 0}$ via $U$, 
	but in contrast to $(e^{tA})_{t \ge 0}$ it does not have any positivity properties, in general.}
	Moreover, $0$ is in the essential range of $m$ since we assumed $0 = \spb(A)$.
	
	Let $\varepsilon > 0$ and let $Q$ denote the spectral projector of $A$ 
	associated with the real interval $[-\varepsilon,0]$, 
	i.e., $Q := U^* M_{\one_S} U$, 
	where $S := m^{-1}([-\varepsilon, 0]) \subseteq \Omega$. 
	Note that $S$ has non-zero measure since $0$ is in the essential range of $m$, 
	and hence $Q$ is non-zero.

	As $w$ is a quasi-interior point of $X^+$ and $Q$ is self-adjoint,
	Proposition~\ref{prop:char-quasi-interior-bilinear} 
	shows the existence 
	of a (non-zero) vector $\tilde w \in [0, w]$ such that $\inner{\tilde w}{Q\tilde w} \not= 0$, 
	and thus $\inner{\tilde w}{Q\tilde w} = \inner{Q\tilde w}{Q\tilde w} > 0$ (as $\tilde Q$ is a self-adjoint projection).
	For all $t \ge 0$ we have
	\begin{align*}
		\inner{v}{e^{tA}u} 
		\ge 
		\inner{\tilde w}{e^{tA} \tilde w} 
		\ge 
		e^{-t\varepsilon} 
		\inner{ \tilde w}{ Q\tilde w },
	\end{align*}
	where the first inequality follows from the positivity of the semigroup 
	and the assumption $H^+ \subseteq (H')^+$ 
	(which implies that $\inner{v}{\argument}$ and $\inner{\tilde w}{\argument}$ are positive functionals 
	and that the former dominates the latter), 
	and where the second inequality follows 
	from the spectral theorem for $A$ and its generated semigroup, i.e., from
	\begin{align*}
		\inner{\tilde w}{e^{tA} \tilde w} 
		& = 
		\inner{U \tilde w}{e^{tm} U \tilde w}_{L^2(\Omega,\mu)} 
		\\ 
		& = 
		\int_\Omega e^{t m(\omega)} \modulus{U \tilde w(\omega)}^2 \dx \mu(\omega)
		\\ 
		& \ge 
		\int_\Omega e^{-t\varepsilon} \one_S(\omega) \modulus{U \tilde w(\omega)}^2 \dx \mu(\omega) 
		\\
		& = 
		e^{-t\varepsilon} \inner{U \tilde w}{M_{\one_S} {U \tilde w}}_{L^2(\Omega,\mu)} 
		= 
		e^{-t\varepsilon} \inner{\tilde w}{Q \tilde w}.
	\end{align*}
	As $\inner{\tilde w}{Q\tilde w} > 0$, we conclude that 
	\begin{align*}
		\liminf_{t \to \infty} \frac{\log \inner{v \,}{\, e^{tA} u}}{t} 
		\ge 
		-\varepsilon,
	\end{align*}
	which proves the claim.
\end{proof}

As pointed out in \cite[Example~2.4]{KellerLenzVogtWojciechowski2015}, 
the self-adjointness in Theorem~\ref{thm:log-ordered-hilbert} cannot be dropped, 
even on $L^2$-spaces with their usual order.
Theorem~\ref{thm:log-ordered-hilbert} can be used to derive a number of interesting consequences.
The first one is another explicit formula for $\spb(A)$. 

\begin{corollary}
	\label{cor:s-a-hilbert}
	Let $(e^{tA})_{t \ge 0}$ be a positive and self-adjoint $C_0$-semigroup 
	with generator $A$ on a non-zero ordered Hilbert space $(H,H^+)$, 
	assume that $H^+ \subseteq (H')^+$ 
	and that $H^+$ contains a quasi-interior point.
	Then 
	\begin{align*}
		-\infty < \spb(A) = \omega(A) = 
		\inf 
		\big\{
			\lambda \in \R: \, 
			\exists u \in \qintt(H^+) \cap \dom{A} \text{ s.t. } Au \le \lambda u.
		\big\}
		,
	\end{align*}
	where $\qintt(H^+)$ denotes the set of quasi-interior points in $H^+$.
\end{corollary}

\begin{proof}
	We only need to show the equality on the right hand side.
	
	``$\le$''
	Let $\lambda \in \R$ and let $u \in \qintt(H^+) \cap \dom{A}$ such that $Au \le \lambda u$.
	We need to show that $\omega(A) \le \lambda$.
	
	By the same argument as in the proof of 
	\Implies{thm:stability-for-pos-sg-interior-point:item:strictly-decreasing-3}{thm:stability-for-pos-sg-interior-point:item:strongly-stable} 
	in Theorem~\ref{thm:stability-for-pos-sg-interior-point},
	one can see that the inequality $Au \le \lambda u$ 
	together with the positivity of the semigroup implies that
	\begin{align*}
		e^{tA} u \le e^{\lambda t} u
	\end{align*}
	for all $t \in [0,\infty)$. 
	By applying Theorem~\ref{thm:log-ordered-hilbert} (with $u=v=w$) we thus obtain
	\begin{align*}
		\omega(A) 
		= 
		\lim_{t \to \infty} \frac{\log \inner{e^{tA} u}{u}}{t}
		\le 
		\lim_{t \to \infty} \frac{\log \big(e^{t \lambda} \norm{u}^2 \big)}{t}
		= 
		\lambda.
	\end{align*}
	
	``$\ge$''
	Let $\lambda > \spb(A)$; 
	we show that $\lambda$ is an element of the set under the infimum. 
	To this end, let $v \in H^+$ be a quasi-interior point (which exists by assumption), 
	and set $u := \Res(\lambda,A) v$.
	Since the resolvent operator $\Res(\lambda,A)$ is positive and has dense range, 
	it follows that $u$ is also a quasi-interior point of $H^+$ \cite[Proposition~2.21]{GlueckWeber2020}.
	Moreover, $u \in \dom{A}$. 
	Finally we note that
	\begin{align*}
		A u = A\Res(\lambda,A) v = \lambda u - v \le \lambda u.
	\end{align*}
	This concludes the proof.
\end{proof}

We point out that the same formula as in Corollary~\ref{cor:s-a-hilbert} is true 
if $(X,X^+)$ is an arbtirary ordered Banach space whose cone is normal and has non-empty interior 
(in particular, without any self-adjointness assumption on the semigroup); 
this was proved in \cite[Theorem~5.3]{ArendtChernoffKato1982}
(note that, if the cone has non-empty interior, the quasi-interior points and the interior points 
coincide \cite[Corollary~2.8]{GlueckWeber2020}).

Another interesting consequence is the fact that, for self-adjoint positive semigroups, 
there is no eigenvalue different from $\spb(A)$ which has a quasi-interior point as its eigenvector.

\begin{corollary}
	\label{cor:s-a-hilbert-eigenvector}
	Let $(e^{tA})_{t \ge 0}$ be a positive and self-adjoint $C_0$-semigroup 
	with generator $A$ on a non-zero ordered Hilbert space $(H,H^+)$, 
	assume that $H^+ \subseteq (H')^+$. 
	
	Let $\lambda \in \R$ be an eigenvalue of $A$ and 
	assume that there exists a corresponding eigenvector $u \in \dom{A}$ 
	which is a quasi-interior point of $H^+$.
	Then $\lambda = \spb(A)$.
\end{corollary}

\begin{proof}
	Since $Au = \lambda u \le \lambda u$, 
	the number cannot be smaller than $\spb(A)$ according to Corollary~\ref{cor:s-a-hilbert}.
\end{proof}

\appendix 

\section{Stability of (non-positive) semigroups on Hilbert spaces}
\label{section:hilbert-non-positive}

In this appendix we give a, to the best of our knowledge new, 
characterization of uniform exponential stability of general $C_0$-semigroups 
(without any positivity assumption) on complex Hilbert spaces.
We will show that one of the implications in the theorem 
is a consequence of a classical result of Gearhart--Prüß--Greiner,
and that the other implication can either be obtained by using a solution 
to Lyapunov's equation~\eqref{eq:lyapunov-alt} below
or by employing our Theorem~\ref{thm:stability-for-pos-sg}.

Let $H$ be a complex Hilbert space, endowed with an inner product $\inner{\cdot}{\cdot}$.
While no order structure is assumed, a connection to positivity might not be too surprising to readers familiar with the characterization of stability in terms of Lyapunov's equation:

If $A$ generates a $C_0$-semigroup on $H$, the growth bound of the semigroup satisfies $\omega(A) < 0$ if and only if there exists a self-adjoint operator $P \in \calL(H)$ 
which is positive semi-definite and injective (i.e., $\inner{Px}{x}>0$ for $x\in H\backslash\{0\}$) 
and which solves Lyapunov's equation
\begin{align}
	\label{eq:lyapunov}
	\inner{A x}{Px} + \inner{Px}{Ax} = - \inner{x}{x}
	\qquad \text{for all } x \in \dom{A};
\end{align}
see \cite[Theorem~5.1.3]{CuZ95}. 
This is intrinsically related to the mappings $\calL(H) \ni S \mapsto e^{tA} S (e^{tA})^* \in \calL(H)$, which leave the positive cone in the self-adjoint part of $\calL(H)$ invariant and constitutes thus a positive semigroup (though this semigroup will not be strongly continuous, in general).
Clearly,~\eqref{eq:lyapunov} is equivalent to 
\begin{align}
	\label{eq:lyapunov-alt}
	\re\inner{A x}{Px} = - \frac{1}{2}
	\qquad \text{for all } x \in \dom{A} \text{ of norm } \|x\|_H=1.
\end{align}
The following theorem is the main result of this appendix.
We will show that the implication 
\Implies{thm:stability-hilbert:itm:stable}{thm:stability-hilbert:itm:estimate}
can either be obtained by directly employing~\eqref{eq:lyapunov}, 
or by following an approach related to the proof~\eqref{eq:lyapunov} -- 
where we will, however, consider the action of the operators $S \mapsto e^{tA} S (e^{tA})^*$ 
on the space $\calK(H)$ of compact operators on $H$;
considering the action on $\calK(H)$ rather than on $\calB(H)$ will give us strong continuity 
with respect to the time variable.

\begin{theorem}
	\label{thm:stability-hilbert}
	Let $A$ be the generator of a $C_0$-semigroup on a complex Hilbert space $H$. 
	The following are equivalent:
	\begin{enumerate}[label=\upshape(\roman*)]
		\item\label{thm:stability-hilbert:itm:stable} 
		The growth bound of the semigroup satisfies $\omega(A) < 0$.
		
		\item\label{thm:stability-hilbert:itm:estimate} 
		There exists a number $\eta > 0$ with the following property: 
		for each $x \in \dom{A}$ of norm $\norm{x}_H = 1$ there is vector $y \in H$ of norm $\norm{y}_H = 1$ which satisfies
		\begin{align*}
			\re \Big( \inner{y}{x} \inner{Ax}{y} \Big) \le -\eta.
		\end{align*}
	\end{enumerate}
\end{theorem}

The advantage that we see in condition~\ref{thm:stability-hilbert:itm:estimate} in the theorem compared to Lyapunov's equation~\eqref{eq:lyapunov-alt} is that vector $y$ in condition~\ref{thm:stability-hilbert:itm:estimate} is not required to depend linearly (or in any other structured way) on $x$.
Assertion~\ref{thm:stability-hilbert:itm:estimate} should also be compared to the stronger energy estimate
\begin{align*}
	\re \inner{Ax}{x} \le -\eta
	\qquad \text{for all } x \in \dom{A} \text{ of norm } \norm{x}_H = 1,
\end{align*}
which is equivalent to the ``quasi-contractive'' exponential stability condition $\norm{e^{tA}} \le e^{-t\eta}$ for all $t \ge 0$.

Let us first show that the proof of the implication 
\Implies{thm:stability-hilbert:itm:estimate}{thm:stability-hilbert:itm:stable} 
in the theorem is a consequence of a classical theorem on the stability of $C_0$-semigroups on Hilbert spaces:

\begin{proof}[Proof of \Implies{thm:stability-hilbert:itm:estimate}{thm:stability-hilbert:itm:stable} in %
				Theorem~\ref{thm:stability-hilbert}] 
	Assume that $\omega(A) \ge 0$.
	Then it follows from the Gearhart--Prüß--Greiner theorem \cite[Theorem~V.1.11]{EngelNagel2000} 
	that there exists a sequence of complex numbers $(\lambda_n)$ in the resolvent set of $A$ 
	such that $\re \lambda_n > 0$ for each $n$ and $\norm{\Res(\lambda_n,A)} \to \infty$. 
	We can thus choose normalized vectors $z_n \in H$ such that $\alpha_n := \norm{\Res(\lambda_n, A) z_n}_H \to \infty$, and we define vectors
	\begin{align*}
		x_n := \frac{1}{\alpha_n} \Res(\lambda_n, A) z_n \in \dom{A}
	\end{align*}
	of norm $\norm{x_n}_H = 1$. 
	For each $x_n$, choose a normalized vector $y_n$ as in assertion~\ref{thm:stability-hilbert:itm:estimate}
	of the theorem.
	Then
	\begin{align*}
		A x_n 
		= 
		-\frac{1}{\alpha_n} z_n + \lambda_n x_n,
	\end{align*}
	and thus
	\begin{align*}
		-\eta 
		\ge 
		\re \Big( \inner{y_n}{x_n} \inner{Ax_n}{y_n} \Big)
		& = 
		\re \Big( -\frac{1}{\alpha_n} \inner{y_n}{x_n} \inner{z_n}{y_n}  +  \lambda_n \inner{y_n}{x_n} \inner{x_n}{y_n} \Big) 
		\\
		& \ge 
		\frac{- \re \Big(\inner{y_n}{x_n} \inner{z_n}{y_n}\Big)}{\alpha_n}
		\to 0,
	\end{align*}
	which is a contradiction.
\end{proof}

We will now give two different proofs for the converse implication 
\Implies{thm:stability-hilbert:itm:stable}{thm:stability-hilbert:itm:estimate}
in Theorem~\ref{thm:stability-hilbert}.
The first one is essentially an application of Lyapunov's equation~\eqref{eq:lyapunov-alt}:
	
\begin{proof}[Proof of \Implies{thm:stability-hilbert:itm:stable}{thm:stability-hilbert:itm:estimate} in %
				Theorem~\ref{thm:stability-hilbert} via Lyapunov's equation] 
	Since the inequality $\omega(A) < 0$ holds,
	there exists a positive semi-definite and injective operator $P \in \calL(H)$
	that satisfies Lyapunov's equation~\eqref{eq:lyapunov-alt}.
	
	In order to show~\ref{thm:stability-hilbert:itm:estimate}, 
	let $x \in \dom{A}$ have norm $1$.
	We choose $y := Px / \|Px\|_H$ 
	(which is well-defined since the denominator is non-zero due to the injectivity of $P$),
	and with this choice
	\begin{align*}
		\re \Big( \inner{y}{x} \inner{Ax}{y} \Big) 
		&= 
		\frac{1}{\|Px\|_H^2}\re \Big( \inner{Px}{x} \inner{Ax}{Px} \Big) \\
		&= 
		\frac{\inner{Px}{x}}{\|Px\|_H^2}\re \Big(  \inner{Ax}{Px} \Big)
		= 
		- \frac{\inner{Px}{x}}{2\|Px\|_H^2}.
	\end{align*}
	In order to find an upper estimate for the latter term,
	define $Q := P/\norm{P}$. 
	So the self-adjoint operator $Q$ is also positive semi-definite,
	and it has norm $1$. 
	For every non-zero $x \in H$ we have
	\begin{align*}
		- \frac{\inner{Px}{x}}{2\|Px\|_H^2} 
		=
		- \frac{1}{\norm{P}} \frac{\inner{Qx}{x}}{2\|Qx\|_H^2}.
	\end{align*}
	Now we use the multiplier version of the spectral theorem for self-adjoint operators:
	it allows us to represent $x$ as an $L^2$-function $f$ (over a suitable measure space $\Omega$) and $Q$ is the multiplication with a real-valued $L^\infty$ function $m$ that takes values in $[0,1]$. So
	\begin{align*}
		- \frac{1}{\norm{P}} \frac{\inner{Qx}{x}}{2 \|Qx\|_H^2}
		=
		- \frac{1}{\norm{P}} \frac{\int_\Omega m \modulus{f}^2}{2 \int_\Omega m^2 \modulus{f}^2} 
		\le 
		- \frac{1}{2 \norm{P}};
	\end{align*}
	the last inequality follows from that fact that $m$ takes values in $[0,1]$ only -- 
	this implies that $m^2 \le m$, so the function under the integral in the denominator 
	is (pointwise) smaller than the function under the integral in the numerator.
	So we proved that assertion~\ref{thm:stability-hilbert:itm:estimate} is satisfied 
	with $\eta = \frac{1}{2 \norm{P}}$.
\end{proof}

Finally, we give a second proof of the implication 
\Implies{thm:stability-hilbert:itm:stable}{thm:stability-hilbert:itm:estimate} in Theorem~\ref{thm:stability-hilbert} 
-- but this time we employ our Theorem~\ref{thm:stability-for-pos-sg} about positive semigroups for the proof.
Readers familiar with the proof of Lyapunov's equality will probably not be too surprised
about the main approach in the proof;
however, we find it worthwhile to give all arguments in detail anyway, 
since they provide an interesting relation to the theory of positive $C_0$-semigroups on ordered Banach spaces.

We will need the following lemma. 
Let $\calK(H)$ denote the $C^*$-algebra of compact linear operators on $H$ and let $\calK(H)_\sa$ denote its self-adjoint part. 
As usual, we denote by $\calK(H)_\sa^+$ the cone of those operators $K$ in $X$ that satisfy $\spec(K) \subseteq [0,\infty)$. 
Then $(\calK(H)_\sa, \calK(H)_\sa^+)$ is an ordered Banach space with normal and generating cone, and $\calK(H)$ is a complexification of $\calK(H)_\sa$. 
We call an operator \emph{positive definite} if its spectrum is contained in $(0,\infty)$.

We have the following formula for the distance to the positive cone in $X$:

\begin{lemma}
	\label{lem:dist-cone-hilbert}
	Let $H$ be a complex Hilbert space. 
	For every self-adjoint compact linear operator $K$ on $H$ that is not positively definite%
	\footnote{The latter assumption is automatically satisfied due to the compactness of $K$ if $H$ is infinite-dimensional, as in this case $0 \in \sigma(K)$.}%
	,
	one has
	\begin{align*}
		\dist(K,\calK(H)_\sa^+) = \sup \Big\{ -\inner{Ky}{y} : \, y \in H, \; \norm{y}_H = 1 \Big\}.
	\end{align*}
\end{lemma}

\begin{proof}
	``$\ge$'': 
	For every operator $L \in \calK(H)_\sa^+$ 
	and every vector $y \in H$ of norm $1$ one has
	\begin{align*}
		\norm{K - L} \ge \inner{(L-K)y}{y} \ge -\inner{Ky}{y}.
	\end{align*}
	
	``$\le$'':
	Let $\lambda \in \R$ denote the number on the right-hand side of the claimed equality. 
	Then $-\lambda$ is the minimum of $\spec(K)$, and $-\lambda \le 0$ since $K$ is not positively definite. 
	As a consequence of the spectral theorem for self-adjoint compact operators 
	we can split $K$ as $K = K^+ - K^-$ for operators $K^+, K^- \in \calK(H)_\sa^+$ 
	such that the spectral radius of $K^-$ is equal to $\lambda$.
	Hence,
	\begin{align*}
		\dist(K,\calK(H)_\sa^+) \le \norm{K - K^+} = \norm{K^-} = \lambda.
	\end{align*}
	This proves the claim.
\end{proof}

Now we give our second proof of the implication 
\Implies{thm:stability-hilbert:itm:stable}{thm:stability-hilbert:itm:estimate}
in Theorem~\ref{thm:stability-hilbert}.
We use the following notation: 
for all $x,y \in H$ the symbol $x \otimes y \in \calL(H)$ denotes the rank-$1$ operator on $H$ given by $(x \otimes y) z = \inner{z}{y} x$ for all $z \in H$;
it has operator norm $\norm{x} \norm{y}$.

\begin{proof}[Proof of \Implies{thm:stability-hilbert:itm:stable}{thm:stability-hilbert:itm:estimate} %
				in Theorem~\ref{thm:stability-hilbert} via positive semigroups]
	For each $t \in [0,\infty)$ consider the operator $T_t$ on $\calK(H)_\sa$ that is given by
	\begin{align*}
		T_t(K) = e^{tA} K (e^{tA})^* 
		\qquad \text{for all } K \in \calK(H)_\sa.
	\end{align*}
	This operator is positive. 
	The family $(T_t)_{t \ge 0}$ clearly satisfies the semigroup law.
	Moreover, it is strongly continuous, as can be seen be first checking strong continuity on finite rank-operators $K$ and then using a density argument. 
	Hence, $(T_t)_{t \ge 0}$ is a positive $C_0$-semigroup on the ordered Banach space $(\calK(H)_\sa, \calK(H)_\sa^+)$; 
	let us denote its generator by $L$. 
	
	It follows from $\omega(A) < 0$ that $\omega(L) < 0$, and hence $\spb(L) < 0$. 
	So we can apply Theorem~\ref{thm:stability-for-pos-sg} and part~\ref{thm:stability-for-pos-sg:itm:uniform-small-gain} of the theorem 
	tells us that here exists a number $\eta > 0$ such that
	\begin{align*}
		\dist(L(K),\calK(H)_\sa^+) \ge \eta \norm{K} \quad \text{for all } 0 \le K \in\dom{L}.
	\end{align*}
	Now let $x \in \dom{A}$ of norm $\norm{x}_H = 1$ and consider the rank-$1$ operator $x \otimes x \in \calK(H)_\sa^+$; 
	it has norm $1$. 
	Moreover, it is easy to check that this operator is in $\dom{L}$ and that we have
	$
		L(x \otimes x) = (Ax) \otimes x + x \otimes (Ax).
	$
	Hence, 
	\begin{align*}
		\dist\big((Ax) \otimes x + x \otimes (Ax), \calK(H)_\sa^+\big) \ge \eta.
	\end{align*}
	
	By Lemma~\ref{lem:dist-cone-hilbert} we thus find an operator a normalised vector $y \in H$ such that 
	\begin{align*}
		-\Big( \, [(Ax) \otimes x + x \otimes (Ax)]y \, \Big| \, y \,\Big) \ge \eta/2
	\end{align*}
	This means precisely that $\re \Big( \inner{y}{x} \inner{Ax}{y} \Big) \le -\eta/4$.
\end{proof}

\section*{Acknowledgements} 

We are indebted to Delio Mugnolo for his suggestion to consider potential generalizations of 
the Collatz--Wielandt formula to infinite dimensions, 
and to Joachim Kerner for helpful discussions regarding the operator $\Delta + V$ 
in Example~\ref{exa:neumann-laplace-with-potential}. 
We also thank Sahiba Arora for pointing out several corrections to us.

\section*{Declarations}

\subsection*{Funding}

This research has been supported by the German Research Foundation (DFG) via the grant MI 1886/2-2.

\subsection*{Conflict of interest/Competing interests}

		There is no conflict of interests.

\subsection*{Ethics approval}

		Not applicable

\subsection*{Availability of data and materials}

		Not applicable

\subsection*{Authors' contributions}

	Both authors contributed to the derivations and proofs of the main
results. Both authors read and approved the final manuscript.

\bibliographystyle{plain}
\bibliography{literature}

\begin{thebibliography}{10}

\bibitem{AliprantisTourky2007}
Charalambos~D. {Aliprantis} and Rabee {Tourky}.
\newblock {\em {Cones and duality}}, volume~84.
\newblock Providence, RI: American Mathematical Society (AMS), 2007.

\bibitem{Arendt1987}
Wolfgang {Arendt}.
\newblock {Resolvent positive operators.}
\newblock {\em {Proc. Lond. Math. Soc. (3)}}, 54:321--349, 1987.

\bibitem{ArendtBattyHieberNeubrander2011}
Wolfgang {Arendt}, Charles J.~K. {Batty}, Matthias {Hieber}, and Frank
  {Neubrander}.
\newblock {\em {Vector-valued Laplace transforms and Cauchy problems}},
  volume~96.
\newblock Basel: Birkh\"auser, 2011.

\bibitem{ArendtChernoffKato1982}
Wolfgang {Arendt}, Paul~R. {Chernoff}, and Tosio {Kato}.
\newblock {A generalization of dissipativity and positive semigroups}.
\newblock {\em {J. Oper. Theory}}, 8:167--180, 1982.

\bibitem{ArendtNittka2009}
Wolfgang {Arendt} and Robin {Nittka}.
\newblock {Equivalent complete norms and positivity.}
\newblock {\em {Arch. Math.}}, 92(5):414--427, 2009.

\bibitem{AroraGlueck2022}
Sahiba Arora and Jochen Gl{\"u}ck.
\newblock Stability of (eventually) positive semigroups on spaces of continuous
  functions.
\newblock {\em C. R., Math., Acad. Sci. Paris}, 360:771--775, 2022.

\bibitem{AroraGlueckPaunonenSchwenningerPreprint}
Sahiba {Arora}, Jochen {Gl\"uck}, Lassi {Paunonen}, and Felix~L.
  {Schwenninger}.
\newblock Limit-case admissibility for positive infinite-dimensional systems.
\newblock 2024.
\newblock Preprint.

\bibitem{BarbieriEngel2024Preprint}
Alessio Barbieri and Klaus-Jochen Engel.
\newblock On structured perturbations of positive semigroups.
\newblock 2024.
\newblock Preprint; available online at arxiv.org/abs/2405.18947v1.

\bibitem{BatkaiJacobVoigtWintermayr2018}
Andr{\'a}s B{\'a}tkai, Birgit Jacob, J{\"u}rgen Voigt, and Jens Wintermayr.
\newblock Perturbations of positive semigroups on {AM}-spaces.
\newblock {\em Semigroup Forum}, 96(2):333--347, 2018.

\bibitem{BatkaiKramarRhandi2017}
Andr\'as {B\'atkai}, Marjeta {Kramar Fijav\v{z}}, and Abdelaziz {Rhandi}.
\newblock {\em {Positive operator semigroups. From finite to infinite
  dimensions}}, volume 257.
\newblock Basel: Birkh\"auser/Springer, 2017.

\bibitem{BattyDavies1983}
Charles J.~K. {Batty} and Edward~B. {Davies}.
\newblock {Positive semigroups and resolvents}.
\newblock {\em {J. Oper. Theory}}, 10:357--363, 1983.

\bibitem{BattyRobinson1984}
Charles J.~K. {Batty} and Derek~W. {Robinson}.
\newblock {Positive one-parameter semigroups on ordered Banach spaces.}
\newblock {\em {Acta Appl. Math.}}, 2:221--296, 1984.

\bibitem{Bonsall1958}
Frank~F. Bonsall.
\newblock Linear operators in complete positive cones.
\newblock {\em Proc. Lond. Math. Soc. (3)}, 8:523--575, 1958.

\bibitem{ChillTomilov2007}
Ralph {Chill} and Yuri {Tomilov}.
\newblock {Stability of operator semigroups: ideas and results.}
\newblock In {\em {Perspectives in operator theory. Papers of the workshop on
  operator theory, Warsaw, Poland, April 19--May 3, 2004}}, pages 71--109.
  Warsaw: Polish Academy of Sciences, Institute of Mathematics, 2007.

\bibitem{Con90}
John~B Conway.
\newblock {\em A course in functional analysis}.
\newblock Springer-Verlag, New York, 1990.

\bibitem{CuZ20}
Ruth Curtain and Hans Zwart.
\newblock {\em Introduction to Infinite-Dimensional Systems Theory: A
  State-Space Approach}.
\newblock Springer, 2020.

\bibitem{CuZ95}
Ruth~F. Curtain and Hans Zwart.
\newblock {\em An introduction to infinite-dimensional linear systems theory},
  volume~21 of {\em Texts Appl. Math.}
\newblock New York, NY: Springer-Verlag, 1995.

\bibitem{DanersGlueckKennedy2016}
Daniel Daners, Jochen Gl{\"u}ck, and James~B. Kennedy.
\newblock Eventually positive semigroups of linear operators.
\newblock {\em J. Math. Anal. Appl.}, 433(2):1561--1593, 2016.

\bibitem{DRW10}
Sergey Dashkovskiy, Bj\"{o}rn R\"{u}ffer, and Fabian Wirth.
\newblock {Small gain theorems for large scale systems and construction of ISS
  Lyapunov functions}.
\newblock {\em SIAM Journal on Control and Optimization}, 48(6):4089--4118,
  2010.

\bibitem{DonskerVaradhan1975}
Monroe~D. Donsker and S.~R.~Srinivasa Varadhan.
\newblock On a variational formula for the principal eigenvalue for operators
  with maximum principle.
\newblock {\em Proc. Natl. Acad. Sci. USA}, 72:780--783, 1975.

\bibitem{DonskerVaradhan1976}
Monroe~D. Donsker and S.~R.~Srinivasa Varadhan.
\newblock On the principal eigenvalue of second-order elliptic differential
  operators.
\newblock {\em Commun. Pure Appl. Math.}, 29:595--621, 1976.

\bibitem{Eisner2010}
Tanja {Eisner}.
\newblock {\em {Stability of operators and operator semigroups.}}
\newblock Basel: Birkh\"auser, 2010.

\bibitem{ElMennaoui1994}
Omar {ElMennaoui}.
\newblock {Asymptotic behaviour of integrated semigroups}.
\newblock {\em {J. Comput. Appl. Math.}}, 54(3):351--369, 1994.

\bibitem{Emelyanov2007}
Eduard~Yu. {Emel'yanov}.
\newblock {\em {Non-spectral asymptotic analysis of one-parameter operator
  semigroups}}.
\newblock Basel: Birkh\"auser, 2007.

\bibitem{EngelNagel2000}
Klaus-Jochen {Engel} and Rainer {Nagel}.
\newblock {\em {One-parameter semigroups for linear evolution equations.}}
\newblock Berlin: Springer, 2000.

\bibitem{Friedland1990}
Shmuel Friedland.
\newblock Characterizations of the spectral radius of positive operators.
\newblock {\em Linear Algebra Appl.}, 134:93--105, 1990.

\bibitem{Friedland2020}
Shmuel Friedland.
\newblock The {Collatz}-{Wielandt} quotient for pairs of nonnegative operators.
\newblock {\em Appl. Math., Praha}, 65(5):557--597, 2020.

\bibitem{Gantouh2023b}
Yassine~El Gantouh.
\newblock Positivity of infinite-dimensional linear systems.
\newblock 2023.
\newblock Preprint.

\bibitem{Gantouh2024}
Yassine~El Gantouh.
\newblock Well-posedness and stability of a class of linear systems.
\newblock {\em Positivity}, 28(2):20, 2024.
\newblock Id/No 16.

\bibitem{Glueck2016}
Jochen {Gl\"uck}.
\newblock {\em Invariant sets and long time behaviour of operator semigroups}.
\newblock PhD thesis, Universit{\"a}t Ulm, 2016.
\newblock DOI: 10.18725/OPARU-4238.

\bibitem{GlueckKaplin2024}
Jochen Gl{\"u}ck and Michael Kaplin.
\newblock Order boundedness and order continuity properties of positive
  operator semigroups.
\newblock {\em Quaest. Math.}, 47:s153--s168, 2024.

\bibitem{GlM21}
Jochen Gl\"uck and Andrii Mironchenko.
\newblock Stability criteria for positive linear discrete-time systems.
\newblock {\em Positivity}, 25(5):2029--2059, 2021.

\bibitem{GlueckWeber2020}
Jochen {Gl\"uck} and Martin~R. {Weber}.
\newblock {Almost interior points in ordered Banach spaces and the long-term
  behaviour of strongly positive operator semigroups.}
\newblock {\em {Stud. Math.}}, 254(3):237--263, 2020.

\bibitem{GohbergGoldbergKaashoek1990}
Israel {Gohberg}, Seymour {Goldberg}, and Marinus~A. {Kaashoek}.
\newblock {\em {Classes of linear operators. Vol. I}}, volume~49.
\newblock Basel etc.: Birkh\"auser Verlag, 1990.

\bibitem{JTP94}
Zhong-Ping Jiang, A.~R. Teel, and L.~Praly.
\newblock Small-gain theorem for {ISS} systems and applications.
\newblock {\em Mathematics of Control, Signals, and Systems}, 7(2):95--120,
  1994.

\bibitem{KanigowskiKryszewski2012}
Adam {Kanigowski} and Wojciech {Kryszewski}.
\newblock {Perron-Frobenius and Krein-Rutman theorems for tangentially positive
  operators}.
\newblock {\em {Cent. Eur. J. Math.}}, 10(6):2240--2263, 2012.

\bibitem{Karlin1959}
Samuel {Karlin}.
\newblock {Positive operators}.
\newblock {\em {J. Math. Mech.}}, 8:907--937, 1959.

\bibitem{KatsikisPolyrakis2006}
Vasilios Katsikis and Ioannis~A. Polyrakis.
\newblock Positive bases in ordered subspaces with the {Riesz} decomposition
  property.
\newblock {\em Stud. Math.}, 174(3):233--253, 2006.

\bibitem{KMS21}
Christoph Kawan, Andrii Mironchenko, Abdalla Swikir, Navid Noroozi, and Majid
  Zamani.
\newblock A {L}yapunov-based small-gain theorem for infinite networks.
\newblock {\em IEEE Transactions on Automatic Control}, 66(12):5830--5844,
  2021.

\bibitem{KellerLenzVogtWojciechowski2015}
Matthias Keller, Daniel Lenz, Hendrik Vogt, and Rados\l~aw Wojciechowski.
\newblock Note on basic features of large time behaviour of heat kernels.
\newblock {\em J. Reine Angew. Math.}, 708:73--95, 2015.

\bibitem{Koshkin2015}
Sergiy Koshkin.
\newblock Positive semigroups and abstract {Lyapunov} equations.
\newblock {\em Positivity}, 19(1):1--21, 2015.

\bibitem{KrasnoselskiiLifshitsSobolev1989}
Mark~A. {Krasnosel'skii}, Evgeni{\u\i}~A. {Lifshits}, and Aleksandr~V.
  {Sobolev}.
\newblock {\em {Positive linear systems. - The method of positive operators.}}
\newblock Berlin: Heldermann-Verlag, 1989.

\bibitem{Krein1950}
Mark~G. Kre{\u\i}n and Mark~A. Rutman.
\newblock Linear operators leaving invariant a cone in a {B}anach space.
\newblock {\em Amer. Math. Soc. Translation}, 1950(26):128, 1950.

\bibitem{LiJia2021}
Desheng Li and Mo~Jia.
\newblock A dynamical approach to the {P}erron-{F}robenius theory and
  generalized {K}rein-{R}utman type theorems.
\newblock {\em J. Math. Anal. Appl.}, 496(2):Paper No. 124828, 22, 2021.

\bibitem{Marek1992}
Ivo Marek.
\newblock Collatz-{Wielandt} numbers in general partially ordered spaces.
\newblock {\em Linear Algebra Appl.}, 173:165--180, 1992.

\bibitem{MartinezMazon1996}
Josep {Martinez} and Jos{\'e}~M. {Mazon}.
\newblock {\(C_0\)-semigroups norm continuous at infinity.}
\newblock {\em {Semigroup Forum}}, 52(2):213--224, 1996.

\bibitem{Meyer2000}
Carl~D. Meyer.
\newblock {\em Matrix analysis and applied linear algebra}.
\newblock Philadelphia, PA: Society for Industrial {and} Applied Mathematics
  (SIAM), 2000.

\bibitem{MKG20}
Andrii Mironchenko, Christoph Kawan, and Jochen Gl\"uck.
\newblock Nonlinear small-gain theorems for input-to-state stability of
  infinite interconnections.
\newblock {\em Mathematics of Control, Signals, and Systems}, 33:573--615,
  2021.

\bibitem{MunozSarantopoulosTonge1999}
Gustavo~A. {Mu\~noz}, Yannis {Sarantopoulos}, and Andrew {Tonge}.
\newblock {Complexifications of real Banach spaces, polynomials and multilinear
  maps.}
\newblock {\em {Stud. Math.}}, 134(1):1--33, 1999.

\bibitem{Mui2023}
Jonathan Mui.
\newblock Spectral properties of locally eventually positive operator
  semigroups.
\newblock {\em Semigroup Forum}, 106(2):460--480, 2023.

\bibitem{Nagel1986}
Rainer {Nagel}, editor.
\newblock {\em {One-parameter semigroups of positive operators.}}
\newblock Springer, Cham, 1986.

\bibitem{Nussbaum1981}
Roger~D. {Nussbaum}.
\newblock {Eigenvectors of nonlinear positive operators and the linear
  Krein-Rutman theorem.}
\newblock {Fixed point theory, Proc. Conf., Sherbrooke/Can. 1980, Lect. Notes
  Math. 886, 309-330 (1981).}, 1981.

\bibitem{Pazy1983}
Amnon {Pazy}.
\newblock {\em {Semigroups of linear operators and applications to partial
  differential equations.}}, volume~44.
\newblock Springer, 1983.

\bibitem{Penney1976}
Richard~C. Penney.
\newblock Self-dual cones in {Hilbert} space.
\newblock {\em J. Funct. Anal.}, 21:305--315, 1976.

\bibitem{Rue10}
Bj{\"o}rn~S R{\"u}ffer.
\newblock Monotone inequalities, dynamical systems, and paths in the positive
  orthant of {E}uclidean n-space.
\newblock {\em Positivity}, 14(2):257--283, 2010.

\bibitem{Schaefer1960}
Helmut~H. Schaefer.
\newblock Halbgeordnete lokalkonvexe {Vektorr{\"a}ume}. {III}.
\newblock {\em Math. Ann.}, 141:113--142, 1960.

\bibitem{Schaefer1967}
Helmut~H. {Schaefer}.
\newblock {Invariant ideals of positive operators in \(C(X)\). I}.
\newblock {\em {Ill. J. Math.}}, 11:703--715, 1967.

\bibitem{Schaefer1974}
Helmut~H. {Schaefer}.
\newblock {\em {Banach lattices and positive operators}}, volume 215.
\newblock Springer, Berlin, 1974.

\bibitem{SchaeferWolff1999}
Helmut~H. {Schaefer} and Manfred P.~H. {Wolff}.
\newblock {\em {Topological vector spaces}}.
\newblock New York, NY: Springer, 2nd ed. edition, 1999.

\bibitem{SchillingSongVondracek2012}
Ren\'e~L. {Schilling}, Renming {Song}, and Zoran {Vondra\v{c}ek}.
\newblock {\em {Bernstein functions. Theory and applications}}, volume~37.
\newblock Berlin: de Gruyter, 2012.

\bibitem{Simon1976}
Barry Simon.
\newblock The bound state of weakly coupled {Schr{\"o}dinger} operators in one
  and two dimensions.
\newblock {\em Ann. Phys.}, 97:279--288, 1976.

\bibitem{Stern1982}
Ronald~J. {Stern}.
\newblock {A note on positively invariant cones}.
\newblock {\em {Appl. Math. Optim.}}, 9:67--72, 1982.

\bibitem{Thieme1998}
Horst~R. {Thieme}.
\newblock {Remarks on resolvent positive operators and their perturbation}.
\newblock {\em {Discrete Contin. Dyn. Syst.}}, 4(1):73--90, 1998.

\bibitem{vanNeerven1992}
Jan {van Neerven}.
\newblock {\em {The adjoint of a semigroup of linear operators}}, volume 1529.
\newblock Berlin: Springer-Verlag, 1992.

\bibitem{vanNeerven1996}
Jan {van Neerven}.
\newblock {\em {The asymptotic behaviour of semigroups of linear operators.}},
  volume~88.
\newblock Basel: Birkh\"auser, 1996.

\bibitem{Vogt2022}
Hendrik Vogt.
\newblock Stability of uniformly eventually positive {{\(C_0\)}}-semigroups on
  {{\(L_p\)}}-spaces.
\newblock {\em Proc. Am. Math. Soc.}, 150(8):3513--3515, 2022.

\bibitem{Weis1995}
Lutz {Weis}.
\newblock {The stability of positive semigroups on \(L_ p\) spaces.}
\newblock {\em {Proc. Am. Math. Soc.}}, 123(10):3089--3094, 1995.

\bibitem{Weis1998}
Lutz {Weis}.
\newblock {A short proof for the stability theorem for positive semigroups on
  \(L_p(\mu)\).}
\newblock {\em {Proc. Am. Math. Soc.}}, 126(11):3253--3256, 1998.

\bibitem{Wickstead1975}
Anthony~W. Wickstead.
\newblock Compact subsets of partially ordered {Banach} spaces.
\newblock {\em Math. Ann.}, 212:271--284, 1975.

\bibitem{Yosida1980}
Kosaku {Yosida}.
\newblock {\em {Functional analysis. 6th ed}}, volume 123.
\newblock Springer, Berlin, 1980.

\bibitem{ZabreikoSmickih1979}
Pëtr~P. Zabre\u{\i}ko and S.~V. Smickih.
\newblock A theorem of {M}. {G}. {K}re\u{\i}n and {M}. {A}. {R}utman.
\newblock {\em Funktsional. Anal. i Prilozhen.}, 13(3):81--82, 1979.

\end{thebibliography}

\end{document}